\documentclass{llncs}

\usepackage[utf8]{inputenc}
\usepackage[T1]{fontenc}
\usepackage{amsmath,amsfonts,amssymb,amscd}
\usepackage{float}
\usepackage{graphics}
\usepackage{graphicx}
\usepackage{multirow}
\usepackage{algorithmic}
\usepackage{algorithm}

\usepackage{url}
\usepackage{verbatim}

\newcommand{\Z}{\mathbb{Z}} 
\newcommand{\Q}{\mathbb{Q}} 
\newcommand{\R}{\mathbb{R}} 
\newcommand{\F}{\mathbb{F}} 
\newcommand{\C}{\mathbb{C}} 
\newcommand{\K}{K} 
\newcommand{\p}{\mathfrak{p}} 
\newcommand{\pg}{\mathfrak{p}} 
\newcommand{\Pg}{\mathfrak{P}} 
\newcommand{\q}{\mathfrak{q}} 

\newcommand{\Cl}{\operatorname{Cl}}
\newcommand{\OO}{\mathcal{O}}
\newcommand{\OK}{\mathcal{O}_K}

\newcommand{\Nm}{\mathcal{N}}
\newcommand{\Log}{\operatorname{Log}}

\newcommand{\ag}{\mathfrak{a}} 
\newcommand{\bg}{\mathfrak{b}} 
\newcommand{\cg}{\mathfrak{c}} 
\newcommand{\qg}{\mathfrak{q}} 

\newcommand{\disc}{\operatorname{disc}}
\newcommand{\dist}{\operatorname{dist}}
\newcommand{\diag}{\operatorname{diag}}

\newcommand{\vol}{\operatorname{Vol}}

\newtheorem{heuristic}{Heuristic}
\newtheorem{remark*}{Remark}

\renewcommand{\algorithmicrequire}{\textbf{Input:}}
\renewcommand{\algorithmicensure}{\textbf{Output:}}

\title{Subexponential algorithms for finding a short generator of a principal ideal and solving $\gamma$-SVP in $\Q(\zeta_{p^s})$}
\author{Jean-Fran\c{c}ois Biasse}
\institute{University of South Florida\\
Department of Mathematics and Statistics\\
biasse@lix.polytechnique.fr}
\date{}
\begin{document}
\maketitle            

\begin{abstract}

In this paper, we present a heuristic 
algorithm that computes the ideal class group and a generator of a 
principal ideal (PIP) in $\Q(\zeta_{p^s})$ in time $2^{O(n^{1/2+\varepsilon})}$ for $n:= \deg(\Q(\zeta_{p^s})) $ and 
arbitrarily small $\varepsilon>0$. We introduce practical improvements to enhance its 
run time, and we describe a variant that can compute a generator of a principal ideal $I$ with 
$\Nm(I)\leq 2^{n^b}$ 
in time $2^{O\left(n^{a + o(1)}\right)}$ given 
a precomputation of the class group taking time $2^{O(n^{2-3a+\varepsilon})}$ for an arbitrarily small $\varepsilon>0$ 
where $b\leq 7a - 2$ and $\frac{2}{5}<a<\frac{1}{2}$. In particular, this precomputation allows us to solve 
these instances of the PIP in with a run time lower than $2^{O(n^{1/2})}$.

Relying on recent work from Cramer et al.~\cite{CVP_cyclo}, this yields 
an attack on all the cryptographic schemes 
relying on the hardness of finding a short generator 
of a principal ideal of $\Q(\zeta_{p^s})$ such as the homomorphic encryption 
scheme of Vercauteren and Smart~\cite{fre_smart}, 
and the multilinear maps of Garg, Gentry and Halevi~\cite{mult_maps}. This attack (with and without 
precomputation) is asymptotically faster than the one relying on the work of Biasse and Fieker~\cite{biasse_subexp,ANTS_XI} 
which runs in time $2^{O(n^{2/3+\varepsilon})}$ for arbitrarily small $\varepsilon >0$. 

In particular, the public keys of the multilinear maps of Garg, Gentry and Halevi~\cite{mult_maps} satisfy the requirement on the input ideal to 
use our PIP algorithm with precomputation. By using $a = 3/7$, $b=1 + o(1)$, 
and given a precomputation of time $2^{O(n^{5/7+\varepsilon})}$ for arbitrarily small $\varepsilon >0$ on $\Q(\zeta_{p^s})$, 
our algorithm provides a key recovery attack 
in time $2^{O\left(n^{3/7 + o(1)}\right)}$

\end{abstract}

\section{Introduction}

Given an ideal $\ag$ of the maximal order $\OK$ of $K = \Q(\zeta_{2^s})$, we want to decide if 
$\ag$ is principal, and if so, compute $\alpha\in\OK$ such that $\ag = (\alpha)\OK$. This corresponds 
to the Principal Ideal Problem (PIP), which is a fundamental problem in computational number theory. 
The resolution of the PIP in classes of number fields of large degree 
recently received a growing attention due to its connection with cryptosystems based on 
the hardness of finding short generators of principal ideals such as the homomorphic encryption 
scheme of Vercauteren and Smart~\cite{fre_smart}, and the multilinear maps of Garg, Gentry and Halevi~\cite{mult_maps}.

A generator of a principal ideal of the maximal order $\OK$ of $K = \Q(\zeta_{p^s})$ can be found in 
heuristic subexponential time $L_\Delta(2/3+\varepsilon,c)$ for some $c>0$ where 
$$L_{\Delta}(a,b) = e^{(b+o(1))(\ln|\Delta|)^a(\ln\ln|\Delta)^{1-a} }$$
by using an algorithm of Biasse and Fieker~\cite{biasse_subexp,ANTS_XI}. We can also 
find one in quantum polynomial time with an algorithm of Biasse and Song~\cite{SODA16} which 
relies on the hidden subgroup resolution algorithm for $\R^{O(n)}$ of Eisentr\"{a}ger, Hallgren, Kitaev and Song~\cite{STOC2014}. 
To go from an arbitrary generator $\alpha$ to a small one, we need to multiply $\alpha$ by the right unit. 
This is an instance of the Bounded Distance Decoding problem in the lattice of the logarithms of the 
complex embeddings of the elements of $K$. Campbel, Groves and Shepherd~\cite{GCHQ} observed that 
using an LLL reduction of the basis of this lattice consisting of the so-called cyclotomic units~\cite[Chap. 8]{Washington} 
and performing a simple round-off reduction yielded the solution to the problem. This fact was corroborated by Schank 
in a replication study~\cite{Pari_implementation}. Shortly thereafter, Cramer, Ducas, Peikert and Regev~\cite{CVP_cyclo} 
proved that this fact was due to the intrinsic geometric properties of the cyclotomic units and that the LLL reduction 
was not necessary. The draft of Campbel et al. also contained elements on ongoing work on a quantum algorithm 
for solving the PIP which was interrupted when they conjectured that the methods of Eisentr\"{a}ger et al.~\cite{STOC2014} 
would ultimately yield a quantum polynomial time algorithm for solving the PIP. 

These recent developments raised the question of the hardness of the Shortest Vector Problem (SVP) in ideal 
lattices (and principal ideal lattices in particular). According to Cramer et al.~\cite[Sec. 6]{CVP_cyclo}, 
a short generator of a principal ideal in a cyclotomic ring is at least within a factor $e^{\tilde{O}(\sqrt{n})}$ 
of a shortest element. To the best of our knowledge, there is currently no method that can leverage the knowledge of 
a short generator to derive an element with length within a better approximation factor to the first minima. However, 
a short generator of a principal ideal is also a solution to $\gamma$-SVP for $\gamma\in e^{\tilde{O}(\sqrt{n})}$. 

\paragraph{\textbf{Contributions}} We describe an attack running in heuristic complexity $2^{O(n^{1/2+\varepsilon})}$ for 
arbitrarily small $\varepsilon>0$ 
against cryptographic schemes relying on the hardness of finding a short generator of an 
ideal in $\Q(\zeta_{p^s})$, including the homomorphic encryption 
scheme of Smart  and Vercauteren~\cite{fre_smart}, and the multilinear maps of Garg, Gentry and Halevi~\cite{mult_maps}. 
We rely 
on a different ideal class group computation and PIP algorithms than the subsexponential methods of Biasse and Fieker~\cite{ANTS_XI} which 
run in time $2^{O(n^{2/3 + \varepsilon})}$ 
for arbitrarily small $\varepsilon > 0$. We take  
advantage of the small height of the defining polynomial of fields of the form $\Q(\zeta_{p^s})$, and we use a modified 
$\qg$-descent method to solve the PIP.

We also describe practical improvements to our algorithm for computing the ideal class group and solving the PIP that 
do not improve the theoretical complexity, but that have a significant impact on the performances of our methods. 

Finally, we present a PIP resolution method that leverages a precomputation on $\Q(\zeta_{p^s})$ of a higher cost 
than  $2^{O(n^{1/2+\varepsilon})}$. Using this precomputation, we achieve a better heuristic complexity than $2^{O(n^{1/2})}$ when 
solving all subsequent PIP instances for which $\Nm(I)\leq 2^{n^b}$ in $\Q(\zeta_{p^s})$. 
More specifically, by spending a precomputation time in $2^{O(n^{2-3a+\varepsilon})}$ for an arbitrarily small $\varepsilon>0$, we can solve the 
PIP on input ideals $I$ with $\Nm(I)\leq 2^{n^b}$ in time $2^{O\left(n^{a + o(1)}\right)}$ when 
\begin{enumerate}
 \item $b\leq 7a - 2$.
 \item $\frac{2}{5} < a < \frac{1}{2}$. 
\end{enumerate} 
This yields a 
heuristic algorithm for $\gamma$-SVP in principal ideals $I$ of $\Q(\zeta_{p^s})$ satisfying $\Nm(I)\leq 2^{n^b}$ with precomputation on $\Q(\zeta_{p^s})$ 
for $\gamma\in e^{\tilde{O}(\sqrt{n})}$ with a better trade-off approximation/cost than BKZ.  For 
example:
\begin{itemize}
 \item With a precomputation of cost $2^{O(n^{5/7 + \varepsilon})}$ on $\Q(\zeta_{p^s})$, the short generators of principal ideals $I$ of 
$\Q(\zeta_{p^s})$ satisfying $\Nm(I)\leq 2^{n+o(1)}$ can be found in time $2^{O(n^{3/7+o(1)})}$.
\item Given a $2^{O(n^{5/7 + \varepsilon})}$ 
precomputation, 
the private keys of the multilinear maps of Garg, Gentry and Halevi~\cite{mult_maps} can be retrieved 
in time $2^{O(n^{3/7+o(1)})}$ from the corresponding public keys. 
\end{itemize}

\section{Background}

\paragraph{\textbf{Lattices}} A lattice is a discrete additive subgroup of $\R^n$ for some integer $n$. The first 
minima of a lattice $\mathcal{L}$ is defined by $\lambda_1 := \min_{\vec{v}\in\mathcal{L}\setminus \{0\}}\|\vec{v}\|$. 
A basis of $\mathcal{L}$ is a set of linearly independent vectors $\vec{b}_1,\cdots,\vec{b}_k$ such that 
$\mathcal{L} = \Z \vec{b}_1 + \cdots + \Z \vec{b}_k$. The determinant of $\mathcal{L}$ is $\det(\mathcal{L}) = \sqrt{\det(B\cdot B^T)}$ where 
$B = (\vec{b}_i)_{i\leq k}\in\R^{k\times n}$ is the matrix of a basis of $\mathcal{L}$. For a full dimensional lattice 
$\mathcal{L}$, the best upper bound we know on $\lambda_1(\mathcal{L})$ is in $O\left( \sqrt{n}\det(\mathcal{L})^{1/n}\right)$. 
The problem of finding a shortest vector $v\in\mathcal{L}$ is known as the Shortest Vector Problem (SVP), while 
the problem of finding $v\in\mathcal{L}$ such that $\|\vec{v}\|\leq \gamma \lambda_1(\mathcal{L})$ for 
some $\gamma\geq 1$ is known has $\gamma$-SVP. A solution $\vec{v}$ to $\gamma$-SVP satisfies $\|\vec{v}\|\in O\left( \gamma\sqrt{n}\det(\mathcal{L})^{1/n}\right)$. 
 Given the matrix of a basis $A$ as 
input, the LLL algorithm~\cite{LLL} returns a basis $(\vec{b}_i)_{i\leq n}$ such that $\frac{\|\vec{b}_1\|}{\det(\mathcal{L})^{1/n}}\in 2^{O(n)}$ 
in polynomial time in $n$ and $\log(|A|)$. The BKZ algorithm~\cite{Ajtai_BKZ} with block size $k$ returns 
a basis $(\vec{b}_i)_{i\leq n}$ such that $\frac{\|\vec{b}_1\|}{\det(\mathcal{L})^{1/n}}\in O(k^{n/k})$ in 
time $2^{O(k)}\operatorname{Poly}(n,\log(|A|)$. Finally, the HKZ algorithm returns a basis $(\vec{b}_i)_{i\leq n}$ such 
that $\frac{\|\vec{b}_1\|}{\det(\mathcal{L})^{1/n}}\in O(\sqrt{n})$ 
in time $2^{O(n)}$. 

\paragraph{\textbf{Number fields}}  A number field $K$ is a finite extension of $\Q$. Its ring of integers $\OK$ has the structure 
of a $\Z$-lattice of degree $n=[K:\Q]$, and the orders $\OO\subseteq\OK$ are the sublattices of $\OK$ 
which have degree $n$ and which are equipped with a ring structure. A number field 
has $r_1\leq n$ real embeddings $(\sigma_i)_{i\leq r_1}$ and $2r_2$ complex 
embeddings $(\sigma_i)_{r_1 < i \leq 2r_2}$ (coming as $r_2$ pairs of conjugates). The field $K$ is isomorphic to 
$\OK\otimes\Q$ where $\OK$ denotes the ring of integers of $K$. We can embed $K$ in 
$K_\R := K\otimes \R \simeq \R^{r_1}\times \C^{r_2}, $ 
and extend the $\sigma_i$'s to $K_\R$. Let $T_2$ be the Hermitian form on $K_\R$ defined by 
$T_2(x,x') := \sum_i \sigma_i(x)\overline{\sigma_i}(x')$, 
and let $\| x\| := \sqrt{T_2(x,x)}$ be the corresponding $L_2$-norm. The norm of an element $x\in K$ is defined by $\Nm(x) = \prod_i\sigma_i(x)$. 
Let $(\alpha_i)_{i\leq d}$ such that 
$\OK = \oplus_i \Z\alpha_i$, then the discriminant of $K$ is given by $\Delta = \det^2(T_2(\alpha_i,\alpha_j))$. 
The volume of the fundamental domain is $\sqrt{|\Delta|}$, and the size of the input of algorithms 
working on an integral basis of $\OK$ is in $O(\log(|\Delta|))$. In $K = \Q(\zeta_{p^s})$, the degree satisfies 
$[K:\Q] = (p-1)p^{s-1}$ and $\Delta = \pm p^{p^{s-1}(ps - s - 1)}$, therefore $\log(|\Delta|)\sim n\log(n)$ 
and we can express the complexity of our algorithms in terms of $n$ (a choice we made in this paper), 
which makes it easier to compare with other lattice reduction result. However, most of the literature on class group 
computation presents complexities in terms of $\log(|\Delta|)$, which is in general the right value to measure the 
input. For example, it makes no sense to express the complexity with respect to the degree of $K$ in infinite 
classes of quadratic number number fields.

\paragraph{\textbf{Cyclotomic fields}} A cyclotomic field is an extension of $\Q$ of the form $K = \Q(\zeta_N)$ where $\zeta_N = e^{2i\pi/N}$ is 
a primitive $N$-th root of unity. The ring of integers $\OK$ of $K$ is 
$\Z[X]/(\Phi_N(X))$ where $\Phi_N$ is the $N$-th cyclotomic polynomial. When $N$ is a power of two, 
$\Phi_N(X) = X^{N/2} + 1$, and when $N = p^s$ is a power of $p>2$, we have 
$\Phi_N(X) = X^{p^{s-1}(p-1)} + X^{p^{s-1}(p-2)} + \cdots + 1$ (which of course generalizes the case $p = 2$). 
Elements $a\in \OK$ are residues of polynomials in $\Z[X]$ modulo $\Phi_N(X)$, and can be identified 
with their coefficient vectors $\vec{a}\in\Z^{\phi(N)}$ where $\phi(N) = p^{s-1}(p-1)$ is the 
Euler totient of $N$ (and the degree of $\Phi_N(X)$). 

\paragraph{\textbf{The ideal class group}} Elements of the form $\frac{I}{d}$ where $I\subseteq \OK$ is an (integral) ideal of 
the ring of integers of $K$ and $d > 0$ are called fractional ideals. Like the orders of $K$, they have the 
structure of a $\Z$-lattice of degree $n=[K;\Q]$, and they form a multiplicative group $\mathcal{I}$.  Elements of $\mathcal{I}$ 
admit a unique decomposition as a power product of prime ideals of $\OK$ (with 
possibly negative exponents). The norm of 
integral ideals is given by $\Nm(I) := [\OK:I]$, which extends to fractional 
ideals by $\Nm(I/J) := \Nm(I)/\Nm(J)$. The norm of a principal (fractional) ideal agrees with the norm
of its generator $\Nm(x\OK) = |\Nm(x)|$. The principal fractional ideals $\mathcal{P}$ of $K$ are a subgroup of $\mathcal{P}$ 
and ideal class group of $\OK$ is defined by 
$\Cl(\OK) := \mathcal{I}/\mathcal{P}.$ We denote by $[\ag]$ the class of a fractional $\ag$ in $\Cl(\OK)$ and by $h$ the 
cardinality of $\Cl(\OK)$ which is a finite group. In $\Cl(\OK)$ we identify two fractional ideals $\ag,\bg$ if there is $\alpha\in K$ 
such that $\ag = (\alpha)\bg$. This is denoted by $\ag\sim \bg$.

\paragraph{\textbf{Units of $\OK$}} Elements $u\in\OK$ that are invertible in $\OK$ are called units. Equivalently, they 
are the elements $u\in\OK$ such that $(u)\OK = \OK$ and also such that $\Nm(u) = \pm 1$. The unit group of $\OK$ 
where $K$ is a cyclotomic field has rank $r = n/2-1$ and has the form 
$\OK^* = \mu\times \langle\epsilon_1\rangle\times \cdots\times \langle\epsilon_r\rangle$ where $\mu$ are roots of 
unity (torsion units) and the $\epsilon_i$ are non-torsion units. Such $(\epsilon_i)_{i\leq r}$ are called a system 
of fundamental units of $\OK$. Units generate a lattice $\mathcal{L}$ of rank $r$ in $\R^{r+1}$ via the embedding 
$$x\in K\longmapsto \Log(x) := \left( \ln(|\sigma_1(x)|) , \cdots, \ln(|\sigma_{r+1}(x)|)\right)$$
where the complex embeddings $(\sigma_i)_{i\leq n}$ are ordered such that the first $r = n/2$ ones are not 
conjugates of each other. The volume $R$ of $\mathcal{L}$ is an invariant of $K$ called the regulator. 
The regulator $R$ and the class number $h$ satisfy 
$hR = \frac{|\mu|\sqrt{|\Delta|}}{2^{r_1}(2\pi)^{r_2}}\lim_{s\rightarrow 1} \left( (s-1)\zeta_\K (s)\right),$
where $\zeta_\K (s) = \sum_\ag \frac{1}{\Nm(\ag)^s}$ is the usual $\zeta$-function associated to $K$ and 
$|\mu|$ is the cardinality of $\mu$ the group of torsion units. This 
allows us to derive a bound $h^*$ in polynomial time under GRH that satisfies $h^* \leq  hR < 2h^*$ (\cite{BachEulerProd}). 
When $K = \Q(\zeta{p^s})$, logarithm vectors of units of the form $u_j = \frac{\zeta_{p^s}^j - 1}{\zeta_{p^s} - 1}$ for $j\in\Z_{p^s}^*$ 
together with $\mu$ generate a sublattice of $\mathcal{L}$ of index $h^+(p^s)$ where $h^+(N)$ is the class number of the 
maximal real subfield of $\Q(\zeta_N)$~\cite[Lemma 8.1]{Washington}. It is conjectured  
that $h^+(2^s)=1$ (Weber class number problem) and  that 
$h^+(p^s)$ remains bounded for fixed $p$ and increasing $s$ (see~\cite{Pomerance_conjecture}). 

\paragraph{\textbf{Notations}} Throughout this paper, $\|A\| = \max_{i,j}|A_{i,j}|$ denotes the infinite norm of a matrix. 
We denote by $\ln(x)$ the natural logarithm of $x$ and by $\log(x)$ its base-$2$ logarithm. 

\section{High level description of the algorithm}

In cryptosystems based on the hardness of finding a short generator of a principal ideal (in particular~\cite{fre_smart} 
and~\cite{mult_maps}), the secret key is a small generator $g\in K = \Q(\zeta_N)$ of a principal ideal $\ag\subseteq\OK$, and 
the public parameters include a $\Z$-basis of $\ag$. We present a subexponential method for retrieving a generator of an 
input ideal that is asymptotically faster than the current state of the art~\cite{biasse_subexp,ANTS_XI}. Combined with the 
Bounded Distance Decoding method in $\Log(\Z[\zeta_{N}]^*)$ presented in~\cite{CVP_cyclo}, this yields a 
better classical attack against the schemes relying the hardness of finding a short generator in a principal ideal. 
We also show how to leverage the precomputation of a large set of relations between generators of $\Cl(\OK)$ to 
solve the PIP in $\OK$ in time $2^{O(n^{a})}$ for $a < 1/2$. This yields an attack in time $2^{n^{3/7+o(1)}}$ against the 
multilinear maps of Garg, Gentry and Halevi~\cite{mult_maps} with a precomputation of time $2^{n^{5/7+\varepsilon}}$ 
on the field. 

\paragraph{\textbf{Solving the Principal Ideal Problem}}

All subexponential time algorithms for solving the Principal Ideal Problem follow the same high level strategy. They derive from the 
subexponential algorithm for classes of fixed degree number fields~\cite{Buchmann}, which itself is a generalization of the algorithm 
of Hafner and McCurley~\cite{hafner} for quadratic number fields. Let $B>0$ be a smoothness bound and 
$\mathcal{B}:= \{ \text{prime ideals }\pg \text{ with } \Nm(\pg)\leq B\}$. We first need to compute a generating 
set of the lattice $\Lambda$ of all the vectors $(e_1,\cdots,e_m)\in\Z^m$ such that 
\begin{equation}\label{eq:rel}
\exists \alpha\in K, \ \ (\alpha)\OK = \pg_1^{e_1}\cdots\pg_m^{e_m},
\end{equation}
where $m = |\mathcal{B}|$. When $B>12\ln^2|\Delta|$, the classes of ideals in $\mathcal{B}$ generate $\Cl(\OK)$ 
under the GRH~\cite[Th. 4]{BBounds}. 
Therefore, $(\mathcal{B},\Lambda)$ is a presentation of the group $\Cl(\OK)$ and the search for a generating set of the 
relations of the form~\eqref{eq:rel} is equivalent to computing the group structure of $\Cl(\OK)$. Indeed, the morphism 
$$\begin{CD}
\Z^m @>{\varphi}>> \mathcal{I} @>{\pi}>> \Cl(\OK)\\
(e_1, \ldots, e_m) @>>> \prod_i\p_i^{e_i} @>>> \prod_i[\p_i]^{e_i}
\end{CD},$$
is surjective, and the class group  $\Cl(\OK)$ is isomorphic to $\Z^m/\ker(\pi\circ\varphi) = \Z^m/\Lambda.$ 
If we only compute a generating set for a sublattice of $\Lambda$ of finite 
index, then the cardinality of the tentative class group we obtain is a multiple of $h =|\Cl(\OK)|$. This can 
be tested by using an estimate of $hR$ given by the methods of~\cite{BachEulerProd} in polynomial time under the GRH. 
If we are missing relations to generate $\Lambda$, we obtain a multiple of $hR$, and this tells us we need 
to collect more relations. 
Let $\ag$ be an input ideal. We find an extra relation of the form $\ag = (\alpha)\pg_1^{y_1}\cdots\pg_m^{y_m}.$ 
The input ideal $\ag$ is principal, if and only if $\pg_1^{y_1}\cdots\pg_m^{y_m}$ is principal too. Then, assuming we have a generating 
set for the lattice $\Lambda$ of relations of the form~\eqref{eq:rel}, we can rewrite $\pg_1^{y_1}\cdots\pg_m^{y_m}$ as a power-product 
of the relations generating $\Lambda$. More specifically, if the relations $(\alpha_i)\OK = \pg_1^{e_{i,1}}\cdots\pg_m^{e_{i,m}},\ \ \text{for } i\leq m$ 
generate $\Lambda$, then there is $x\in\Z^m$ such that $xA=y$ for $A = (e_{i,j})_{i,j\leq m}$. Then 
$\pg_1^{y_1}\cdots\pg_m^{y_m} = (\alpha_1^{x_1}\cdots\alpha_m^{x_m})\cdot\OK$, which gives us a generator of $\ag$, $\ag = (\alpha\cdot \alpha_1^{x_1}\cdots\alpha_m^{x_m})\cdot\OK.$

\begin{algorithm}[ht]
\caption{Principal Ideal Problem}
\begin{algorithmic}[1]\label{alg:PIP}
\REQUIRE Ideal $\ag\subseteq\OK$, smoothness bound $B$.
\ENSURE $\operatorname{false}$ or $\beta\in K$ such that $\ag = (\beta)\OK$.
\STATE Compute $\mathcal{B} = \left\lbrace \p_1,\hdots,\p_m\right\rbrace$ of norm less than $B$.
\STATE Compute a lattice $\Lambda\subseteq\Z^m$ of random relations between elements of $\mathcal{B}$. 
\STATE $h\leftarrow \left| \Z^m/\Lambda\right|$, $R\leftarrow \text{regulator of }K$.
\STATE Check $h\cdot R$ with the estimates of~\cite{ANTS_XI}[Sec. 4.3]. Find more relations if necessary.
\STATE $A \leftarrow (e_{i,j})_{i,j\leq m}$, where $(\alpha_i)=\prod_i\pg_j^{e_{i,j}}$ are a basis of $\Lambda$.
\STATE Find $\vec{y}\in\Z^m$ such that $\ag = (\alpha)\pg_1^{y_1}\cdots\pg_m^{y_m}.$
\STATE Solve $\vec{x}A=\vec{y}$. 
\RETURN $\operatorname{false}$ if no solution or $\alpha\cdot\alpha_1^{x_1},\hdots,\alpha_{m}^{x_m}$.
\end{algorithmic}
\end{algorithm}

\begin{theorem}
 Algorithm~\ref{alg:PIP} runs in heuristic complexity $2^{O(n^{1/2+\varepsilon})}$ for 
 arbitrarily small $\varepsilon>0$ when 
 $B = 2^{n^{1/2}}$ with $K = \Q(\zeta_{p^s})$ and $n = [K:\Q]$.
\end{theorem}

\paragraph{\textbf{PIP with precomputation on $K$}}

Algorithm~\ref{alg:PIP} can be divided into two steps: the computation of a basis of the lattice of 
relations between the classes of a generating set for $\Cl(\OK)$, and the decomposition of the input 
ideal $\ag$ with respect to this generating set. The computation of relations between generators of the 
ideal class group is independent from the specific instances of the PIP and can be precomputed. Moreover, 
the larger the smoothness bound $B$ is, the faster the decomposition of an input ideal over $\mathcal{B}$ is. 
Naturally, this implies that the precomputation is more expensive since it requires the calculation of more 
relations.

\begin{theorem}
Let $a,b>0$ such that $b\leq 7a-2$ and $\frac{2}{5} < a < \frac{1}{2}$. 
If we precompute all relations between ideals of norm less than $B \in 2^{O(n^{2-3a+\varepsilon})}$ for some some $\varepsilon>0$, 
then the cost of solving the PIP for input ideals $I$ with $\Nm(I)\leq 2^{n^b}$ is in  $2^{O\left(n^{a + o(1)}\right)}$ 
while the cost of the precomputation is in $2^{O(n^{2-3a+\varepsilon})}$. 
\end{theorem}
For example, for $a = 3/7$ and $b = 1 + o(1)$, a precomputation of all relations between prime ideals of norm less than $B \in 2^{O(n^{5/7 + \varepsilon})}$ 
for some $\varepsilon > 0$ 
can be done in time $2^{O(n^{5/7 + \varepsilon})}$. It allows us to solve all instances of the PIP on input $I\subseteq\OK$ such that $\Nm(I)\leq 2^{n^{1+o(1)}}$ in time 
$2^{O(n^{3/7 + o(1)})}$.

\paragraph{\textbf{Reducing the short-PIP and $\gamma$-SVP to the PIP}}

Assume that the input ideal $\ag\subseteq\OK$ is generated by a short element $g$, and that we have computed $\alpha\in\OK$ 
such that $\ag = (\alpha)\cdot\OK$. 
Given a generating set $\gamma_1,\cdots,\gamma_r$ of the unit group $\OK^*$, all 
generators $g'$ of $\ag$ are of the form 
$$g' = \alpha\cdot \gamma_1^{x_1}\cdots\gamma_r^{x_r}\ \ \text{for some } (x_1,\cdots,x_r)\in\Z^r.$$
The problem of finding $g$ (or another short generator, which is equivalent for the sake of a cryptanalysis), 
boils down to finding $(x_1,\cdots,x_r)$ such that $\alpha\cdot \gamma_1^{x_1}\cdots\gamma_r^{x_r}$ is short. In the 
lattice of logarithm embeddings, this can be done by finding $(x_1,\cdots,x_r)$ such that 
$\| \Log(\alpha) - \sum_i x_i \Log(\gamma_i)\|$ is small. To do this, we find the closest vector to $\Log(\alpha)$ 
in the lattice $\mathcal{L}:= \Z\Log(\gamma_1)+\cdots+\Z\Log(\gamma_r)$. 

The closest vector problem is a notoriously hard problem without prior knowledge on the properties of the lattice and the 
target vector. In our situation, things are made easier by the knowledge of the distribution of the target vector (given in 
the description of the cryptosystems~\cite{fre_smart,mult_maps}) and of a good basis for $\mathcal{L}$. Indeed, the decryption 
algorithm of schemes relying on the hardness of the short-PIP works under the assumption that the generator $g$ is small. This 
means that the target vector $\Log(\alpha)$ is very close to the lattice $\mathcal{L}$. Our instance of the Closest Vector Problem 
therefore turns into an instance of the Bounded Distance Decoding problem (BDD) since we have a bound on 
$\dist(\Log(\alpha),\mathcal{L})$. Moreover, we know a very good basis $\Log(\gamma_i)$ for $\mathcal{L}$ with which  
Babai's nearest plane algorithm~\cite{barbai} returns the correct value.

\begin{algorithm}[ht]
\caption{short-PIP to PIP reduction}
\begin{algorithmic}[1]\label{alg:CVP}
\REQUIRE A generator $\alpha\in\OK$ of the ideal $\ag\subseteq\OK$.
\ENSURE A short generator of $\ag$. 
\STATE Let $(\gamma_i)_{i\leq r}$ be the cyclotomic units of $K$.
\STATE $\mathcal{L}\leftarrow \Z\Log(\gamma_1)+\cdots\Z\Log(\gamma_r)$.
\STATE Find the closest vector $\sum_i x_i\Log(\gamma_i)\in\mathcal{L}$ to $\Log(\alpha)$ by using Babai's round-off method. 
\RETURN $\frac{\alpha}{\prod_i\gamma_i^{x_i}}$.
\end{algorithmic}
\end{algorithm}

\begin{theorem}[Th. 4.1 of~\cite{CVP_cyclo}]
When $K = \Q(\zeta_{p^s})$, Algorithm~\ref{alg:CVP} runs in polynomial time.
\end{theorem}
A short generator of a principal ideal in $K = \Q(\zeta_{p^s})$ yields a solution to $\gamma$-SVP for $\gamma\in e^{\tilde{O}(\sqrt{n})}$. 
This can be done in quantum polynomial time using~\cite{SODA16}, or in time $2^{O(n^{a + o(1)})}$ 
given a precomputation of cost in $2^{O(n^{2-3a+\varepsilon})}$ for arbitrarily small $\varepsilon > 0$. 

\begin{theorem}
Let $a,b>0$ such that $b\leq 7a-2$ and $\frac{2}{5} < a < \frac{1}{2}$. 
If we precompute all relations between ideals of norm less than $B \in 2^{O(n^{2-3a+\varepsilon})}$ for some some $0 < \varepsilon$, 
then the cost of finding a short generator (and solving $\gamma$-SVP with $\gamma\in e^{\tilde{O}(\sqrt{n})}$ for input ideals $I$ with $\Nm(I)\leq 2^{n^b}$ is in  $2^{O\left(n^{a + o(1)}\right)}$ 
while the cost of the precomputation is in $2^{O(n^{2-3a+\varepsilon})}$. 
\end{theorem}

\section{Smoothness of ideals}\label{sec:smoothness} 

Given $B>0$, the expected time to find a relation of the form $(\alpha) = \pg_1^{e_1}\cdots\pg_m^{e_m}$ 
where $\Nm(\pg_i)\leq B$ depends on the probability that a random ideal $\ag$ of bounded norm is 
$B$-smooth, that is to say of the form $\ag = \pg_1^{e_1}\cdots\pg_m^{e_m}$. In~\cite{Scourfield}, 
Scourfield established a result on the smoothness of ideals in a number field 
comparable to the ones known on integers. Let 
$$\Psi(x,y) := \left| \{ \ag\subseteq \OK, \Nm(\ag)\leq x, \ag\ y\text{--smooth}\} \right|,$$ 
and $\varepsilon > 0$, then $\frac{\Psi(x,y)}{x}\sim \lambda_K\rho(u)$, where $u=\frac{\ln(x)}{\ln(y)}$, 
$\rho$ is the Dickman function, $\lambda_K$ is the residue of the zeta function $\zeta_K(s)$ at 
$s=1$ and 
$(\ln\ln(x))^{\frac{5}{3}+\varepsilon}\leq \ln(y)\leq \ln(x),\ x \geq x_0(\varepsilon)$ 
for some $x_0(\varepsilon)$. In the case where $K$ is normal and $\frac{n}{\ln|\Delta|}\rightarrow 0$, 
$\lambda_K$ can be bounded absolutely, but there is no such result in the general case. 
During our relation search algorithm, we draw principal ideals at random. There is no known 
analogue of Sourfield's result for restricted classes of ideals. This is one of the reasons why the complexity of the number field 
sieve~\cite{NF_sieve} is only heuristic. 
We therefore 
rely on the following heuristic for the smoothness of ideals. 

\begin{heuristic}\label{heuristic:smoothness_ideals}
We assume that the probability $P(\iota,\mu)$ that a principal ideal of $\OK$ of norm bounded by $\iota$ 
is a power-product of prime ideals of norm bounded by $\mu$ satisfies
\begin{equation}\label{eq:smooth_ideal}
P(\iota,\mu)\geq e^{\left( -u \ln u (1+o(1))\right)},\ \text{for }u = \ln(\iota) / \ln(\mu).
\end{equation}
\end{heuristic}

\section{Computation of $\Cl(\OK)$}\label{sec:CL_group}

In this section, we show how to compute $\Cl(\OK)$ where $K = \Q(\zeta_{p^s})$ is 
a cyclotomic field of prime power conductor in time $L_\Delta(1/2,c)$ for some constant $c>0$ 
and $\Delta = \disc(K)$ 
under Heuristic~\ref{heuristic:smoothness_ideals}. The best known heuristic complexity 
for fields of degree $n\in \tilde{\Theta}(\log(|\Delta|))$ is in $2^{O(n^{2/3+\varepsilon})}$ 
where $\varepsilon>0$ is an  
arbitrarily small constant~\cite{ANTS_XI}. Cyclotomic fields of prime power conductor 
have a defining polynomial 
with height 1, which allows us to use a different technique than the one described 
in~\cite{ANTS_XI}. All we have to do is to draw elements $\alpha\in\OK$ with small 
coefficients on the power basis $1,\zeta_{p^s},\cdots,\zeta_{p^s}^{n-1}$ and test them for smoothness 
with respect to a factor basis $\mathcal{B} = \{ \pg\ \mid \ \Nm(\pg)\leq B\}$ for 
some smoothness bound $B > 0$. The smoothness test is simply done by checking 
if $\Nm(\alpha)$ is $B$-smooth as an integer using either a factoring algorithm~\cite{NF_sieve,pomerance_qs} 
or a dedicated smoothness test algorithm~\cite{bernstein}. Every time we have a relation of the form 
$$(\alpha) = \pg_1^{e_1}\cdots\pg_m^{e_m},$$
we store the vector $(e_1,\cdots,e_m)$ in the rows of a matrix $M$. Once enough relations are found, we 
complete the computation as in all previous subexponential methods~\cite{BSub,HMSub,ANTS_XI} by processing the matrix 
$M$. 

\begin{algorithm}[ht]
\caption{Computation of the class group of $\Q(\zeta_l)$}
\begin{algorithmic}[1]\label{alg:class_group}
\REQUIRE A smoothness bound $B > 0$, a constant $A > 0$ and a conductor $N$. 
\ENSURE $d_i$ such that $\Cl(\OK)=\bigoplus_i \Z/d_i$ for $K = \Q(\zeta_{N})$ and $M\in\Z^{k\times m}$,$(\alpha_i)_{i\leq k}\in\OK^n$ 
such that for each row $M_i$ of $M$, $\mathcal{B}^{M_i} = (\alpha_i)$, where $\mathcal{B}= \{\pg\ \mid \ \Nm(\pg)\leq B\}$. 
\STATE Compute $\mathcal{B} = \{\pg \ \mid \ \Nm(\pg)\leq B\}$. 
\STATE $m\leftarrow |\mathcal{B}|$, $k\leftarrow m$, $M\in\Z^{0\times m}$. 
\WHILE{ The number of relations is less than $k$}
\STATE $(a_i)_{i\leq n}\xleftarrow[]{\mathcal{R}}[-A,A]^n$, $\alpha\leftarrow \sum_i a_i\zeta_N^i$.
\IF{$(\alpha)$ is $\mathcal{B}$-smooth}
\STATE Find $(e_i)_{i\leq m}$ such that $(\alpha)=\prod_i\pg_i^{e_i}$.
\STATE $M\leftarrow \left(\frac{M}{(e_i)}\right)$.
\ENDIF
\ENDWHILE
\STATE \textbf{if} $M$ does not have full rank \textbf{then} $k\leftarrow 2k$ and go to Step~3.
\STATE $H\leftarrow \operatorname{HNF}(M)$. $d\leftarrow \det(H)$. $B\leftarrow \ker(M)$. 
\STATE $L\leftarrow \{\Log(\alpha_1),\cdots,\Log(\alpha_k)\}$. $C\leftarrow \left(L^{B_i}\right)$.
\STATE Let $V$ be the volume of the lattice generated by the rows of $C$, and $h^*$ be an approximation of $hR$ given by the 
methods of~\cite{BachEulerProd}.
\STATE \textbf{if} $dV > 1.5h^*$ \textbf{then} $k\leftarrow 2k$ and go to Step~3.
\STATE $\diag(d_i,\cdots,d_m)\leftarrow \operatorname{SNF}(H)$. 
\RETURN $(d_i)_{i\leq m}$, $M$. $(\alpha_i)_{i\leq k}$.
\end{algorithmic}
\end{algorithm}

The run time of Algorithm~\ref{alg:class_group} depends on the probability of smoothness of 
principal ideals, which is ruled by Heuristic~\ref{heuristic:smoothness_ideals}. This gives 
us a bound on the average time to find a relation. However, there is no indication that the 
relations we find are distributed according to a distribution in $\Lambda$ allowing us to 
terminate the computation in subexponential time. Suppose we found a full rank sublattice $\Lambda_0$ of 
$\Lambda$, Hafner and McCurley~\cite{hafner} 
proved under GRH that their relation search for quadratic fields yielded relation vectors such that 
the probability of drawing one in the the coset $\vec{w} + \Lambda_0$ for any $\vec{w}\in \Lambda$ 
was bounded from below by a large enough bound. This allowed them to justify that the algorithm would 
terminate with high enough probability in subexponential time. It it reasonable to assume that 
by drawing coefficient vectors uniformly at random in $[-A,A]$, the generators of the principal ideals 
of our relations will be well enough distributed to justify that the relations themselves are 
equally distributed in $\Lambda$, but proving it remains an open question. 

\begin{heuristic}\label{heuristic:relation_lattice}
Let $\Lambda_0$ be a sublattice of $\Lambda$ corresponding to the relations between primes 
in $\mathcal{B}$, and $A > 1$ be a constant. If $(a_i)_{i\leq n}$ is 
drawn uniformly at random in $[-A,A]^n$, and if $(\alpha)\OK = \mathcal{B}^{\vec{w}_\alpha}$ is $\mathcal{B}$-smooth, 
then for any $\vec{w}\in\Lambda$ 
$$P\left( \vec{w}_\alpha \in \vec{w} + \Lambda_0\right) \geq \frac{\det(\Lambda)}{\det(\Lambda_0)}\left( 1 + o(1)\right).$$
\end{heuristic}

\begin{proposition}[GRH+Heuristic~\ref{heuristic:smoothness_ideals}+Heuristic~\ref{heuristic:relation_lattice}]\label{prop:class_group}
Algorithm~\ref{alg:class_group} with $B \in 2^{O(n^{1/2})}$ and $N$ of the form $p^s$ 
is correct and its heuristic complexity is in $2^{O(n^{1/2})}$
\end{proposition}

\begin{proof}
The run time depends on the smoothness probability of $\alpha\in\OK$ drawn in Step~4. Let $P\in\Z[X]$ such that 
$\alpha = P(\zeta_{N})$ for $N = p^s$. The norm of $\alpha$ is given by $\operatorname{Res}(\Phi_N,P)$ where 
$\Phi_N$ is the $N$-th cyclotomic polynomial. The first $n$ rows of the resultant have 
length less than $\sqrt{n}$ while the last $n$ rows have length bounded by $\sqrt{n}A$. 
By Hadamard's bound, the resultant is bounded by $n^nA^n$. This means 
that $\log(|\Nm(\alpha)|) \leq n\log(n)(1+o(1))$ (as $A$ is a constant). 
Let $u := \frac{\log(|\Nm(\alpha)|)}{\log(B)}$, from Heuristic~\ref{heuristic:smoothness_ideals}, 
the probability of finding a smooth $\alpha$ is at least 
$e^{-u\ln(u)(1+o(1))} \in \frac{1}{2^{O\left(n^{1/2}\right)}}$, and therefore the 
whole relation search takes time $2^{O(n^{1/2})}$. 
The linear algebra phase (HNF and SNF computation) takes time $|\mathcal{B}|^{4+o(1)} \in 2^{O(n^{1/2})}$, 
which is asymptotically the same as the relation collection phase. 
\end{proof}

\begin{corollary}[GRH+Heuristic~\ref{heuristic:smoothness_ideals}+Heuristic~\ref{heuristic:relation_lattice}]
Algorithm~\ref{alg:class_group} has heuristic complexity $2^{O(n^{a})}$ when $B \in 2^{O(n^{a})}$ for 
$\frac{1}{2}  \leq a < 1$ and $b > 0$. 
\end{corollary}

\section{A $\qg$-descent algorithm to solve the PIP}\label{sec:q_desc}

In this section we show how decompose an input ideal $I\subseteq\OK$ into a 
$B$-smooth product for ideals in $\Cl(\OK)$ for some $B>0$. In other words, 
we find $\pg_1,\cdots,\pg_m$ with $\Nm(\pg_i)\leq B$, $(e_i)_{i\leq m}\in\Z^m$ and 
$\alpha\in K$ such that $I = (\alpha)\prod_i\pg_i^{e_i}$. Then $I$ is principal if 
and only if $\prod_i\pg_i^{e_i}$ is principal as well, which is checked by solving a linear system 
as described in Algorithm~\ref{alg:PIP}. In the case where $I$ is principal, this 
process also returns a generator. 

To decompose $I$, we execute a $\qg$-descent consisting of searching for small 
elements $\alpha$ of $I$ such that $(\alpha)/I$ is smooth with respect to ideals of norm less than 
$I$. Then we iterate the process on all divisors of $(\alpha)/I$ that have norm greater 
than $B$ until we finally break it down into a product of ideals of norm less than 
$B$.

The 
idea of the $\qg$-descent derives from the algorithms based on the number 
field sieve~\cite{NF_sieve} to solve the discrete logarithm problem in time $L_q(1/3)$ in $\F_q$ 
(see in particular~\cite{Adleman93,gordon_GFp,NFS_Fre}). This idea was also used to 
solve the discrete logarithm problem in the Jacobian of certain classes of $C_{ab}$ curves~\cite{Enge2}. 
A $\qg$-descent strategy was used to 
derive relations in $\Cl(\OK)$ in~\cite{biasse_Phd,biasse_subexp,ANTS_XI}, but it was restricted 
to classes of fields with degree $n \in O(\log(|\Delta|)^a$ for some $a < 1/2$. In this paper, we 
achieve a heuristic $L_\Delta(1/2+\varepsilon,1)$ complexity for any $\varepsilon>0$ in 
cyclotomic fields despite the fact that the degree of these fields satisfies $n\in\tilde{\Theta}(\log(|\Delta|))$. 

\paragraph{\textbf{Finding short elements in $I$}}

We assume that $I$ is an ideal of $\OK$. We want to find elements $\alpha\in I$ 
of small norm. To do this, we restrict the search to the lattice 
$$\mathcal{L}_I := \Z v_{1,1} + \Z (v_{2,2}\zeta_N + v_{2,1}) + \cdots + \Z(v_{k,k}\zeta_N^k + v_{k,k-1}\zeta_N^{k-1}+\cdots + v_k)\subseteq I,$$
for some $k > 0$ where $N = q^s$ is the conductor of $K$. This strategy was used in~\cite{biasse_Phd,biasse_subexp,ANTS_XI} 
in the case of $I=\qg$ a degree one prime ideal, which is enough for the sake of collecting relations to compute $\Cl(\OK)$ and 
solve the PIP. However, it was pointed out in~\cite{EUROCRYPT17} that this approach was folklore. In particular, it has been 
used under the more general form presented in this paper by Cheon~\cite{CL15}. 
The coefficients $v_{i,j}$ are given 
by the HNF of the $\Z$-basis of $I$ which has the shape
\[ H:=\left( \begin{array}{ccccc}
v_{1,1}    &  0    & \hdots & \hdots &     0    \\
v_{2,1}    &  v_{2,2}    & \ddots &        &    \vdots \\
\vdots &  0    & \ddots & \ddots &    \vdots \\
\vdots & \vdots& \ddots &  \ddots&      0    \\
v_{N,1}& 0       & \hdots & 0      &   v_{N,N}\\
\end{array} \right).\] 

\begin{lemma}\label{lemma:smooth}
We can find a vector  
in $\mathcal{L}_I$ of length less than $\sqrt{k}\Nm(I)^{\frac{1}{k}}$ in time $2^{O(k)}$.  
\end{lemma}

\begin{proof}
The determinant of $\mathcal{L}_I$ is that of the upper left $k\times k$ 
submatrix of $H$ and satisfies $\det(\mathcal{L}_I) \leq \prod_{i\leq N}v_{i,i} = \Nm(I)$. 
An HKZ-reduction returns a basis 
whose first vector has length less than $\sqrt{k}(\det(\mathcal{L}_I))^{1/k}$ in time $2^{O(k)}$. 
\end{proof}

\begin{lemma}\label{lemma:smooth_combinations}
In time $2^{O(k)}$, we can find an element $\alpha\in I$ 
such that $\Nm(\alpha)\leq n^{n}\Nm(I)^{\frac{n}{k}}.$
\end{lemma}

\begin{proof}
Let $\alpha$ be the first vector of an HKZ-reduced basis of $\mathcal{L}_I$. The calculation of this basis 
takes time $2^{O(k)}$ and by Lemma~\ref{lemma:smooth}, the length of its first vector $(\alpha_1,\cdots,\alpha_k)$ is bounded 
by $\sqrt{k}\Nm(I)^{\frac{1}{k}}$. 
By the same argument as in the proof of Proposition~\ref{prop:class_group}, 
the algebraic norm of $\alpha := \sum_i \alpha_i \zeta_N^i$ satisfies 
$$\Nm(\alpha) \leq \left.\sqrt{n}\right.^n \left( \|(\alpha_1,\cdots,\alpha_k)\| \right)^n\leq \underbrace{n^{n/2} k^{n/2}}_{\leq n^n}\Nm(I)^{\frac{n}{k}}.$$
\end{proof}

\paragraph{\textbf{A round of the $\qg$-descent}}

Let $\varepsilon > 0$ be an arbitrarily small constant and $\qg$ be a prime 
ideal with $\log(\Nm(\qg)) \leq n^b$ for some $1/2< b \leq 1$. 
We would like to decompose the class of $\qg$ in $\Cl(\OK)$ as a product of primes $\pg$ 
with $\log(\Nm(\pg))\leq n^{b-\varepsilon}$ in time $2^{O(n^{1/2 + \varepsilon})}$. 

\begin{algorithm}[ht]
\caption{One round of the $\qg$-descent}
\begin{algorithmic}[1]\label{alg:one_round}
\REQUIRE Prime ideal $\qg$ with $\log(\Nm(\qg))\leq n^b$ for $1\geq b \geq 1/2 + \varepsilon$, a constant $\varepsilon > 0$ and $A\geq 1$. 
\ENSURE Prime ideals $\pg_i$, integers $e_i$ and $\alpha\in\qg$ such that $(\alpha)/\qg = \prod_i\pg_i^{e_i}$ 
and $\log(\Nm(\pg_i))\leq n^{b-\varepsilon}$. 
\STATE $S\leftarrow \left\{\pg_i\ \text{such that }\Nm(\pg_i)\leq 12\log(|\Delta|)^2\right\}$.
\WHILE{no relation has been found}
\STATE $(x_i)\xleftarrow[]{R}[0,A]^{|S|}$ where $R$ is uniform over vectors of weight $n^{b}$, 
\STATE $I\leftarrow \qg\prod_i \pg_i^{x_i}$. 
\STATE Construct a basis for the lattice $\mathcal{L}_I$ with $k = n^{1/2 + \varepsilon}$. 
\STATE HKZ-reduce $\mathcal{L}_I$ and let $\alpha$ be its first vector. 
\STATE \textbf{if} $(\alpha)$ is $2^{n^{b-\varepsilon}}$-smooth, find $(\pg_i)$,$(e_i)$ such that $(\alpha) = \prod_i\pg_i^{e_i}$.  
\ENDWHILE
\RETURN $\alpha$, $(e_i)$, $(\pg_i)$, $(e_i)$. 	
\end{algorithmic}
\end{algorithm}
\noindent Note that Step~3 and Step~4 of Algorithm~\ref{alg:one_round} were used in~\cite{EUROCRYPT17}. 
These steps can be replaced by an enumeration of 
short vectors in $\qg$. An enumeration strategy was used in~\cite{biasse_Phd,biasse_subexp,ANTS_XI}, 
and in this paper, we introduce a similar variant as a practical improvement (see Section~\ref{sec:improvements}). 
The advantage of using the enumeration strategy of Section~\ref{sec:improvements} is that it requires only 
on HKZ reduction. However, we must assume an extra heuristic on the size of the shortest vector. With 
Step~3 and Step~4 of Algorithm~\ref{alg:one_round} on the other hand, the analysis is simpler.

\begin{proposition}[GRH + Heuristic~\ref{heuristic:smoothness_ideals}]\label{prop:run_time_qstep}
Let $\varepsilon > 0$ be a constant, $1/2+\varepsilon\leq b\leq 1$, 
and $\qg$ be a prime ideal with $\log(\Nm(\qg))\leq n^b$. 
There is a large enough constant $A$ such that Algorithm~\ref{alg:one_round} returns a decomposition of 
$\qg$ in $\Cl(\OK)$ as a product of degree one primes $\pg_i$ with 
$\log(\Nm(\pg))\leq n^{b-\varepsilon}$ in time $2^{O(n^{1/2 + \varepsilon})}$ when 
the conductor $N$ is of the form $p^s$. 
\end{proposition}

\begin{proof}
The ideal $I$ created in Step~3 of Algorithm~\ref{alg:one_round} satisfies $\Nm(I) = 2^{n^{\tilde{O}(b)}}$ as the norm of the extra factor used 
for randomization is $\Nm(\prod_i \pg_i^{x_i}) \leq 2^{\tilde{O}(An^{b})}$. Therefore, 
according to Lemma~\ref{lemma:smooth_combinations}, the $\alpha$ derived in Step~5 
satisfies $\log(\Nm(\alpha)) \leq \frac{n}{k}\log\left(\Nm(I)\right)(1+o(1))$. As $\log\left(\Nm(I)\right) \leq \tilde{O}(n^{b})$ and $k = n^{1/2 + \varepsilon}$, 
this gives us $\log(\Nm(\alpha)) \leq \tilde{O}\left(n^{1/2 + b - \varepsilon}\right)(1+o(1))$. We want $(\alpha)/I$ to be 
$B$-smooth for $\log(B) = n^{b - \varepsilon}$. 
Let $u = \frac{\log(\Nm(\alpha))}{\log(B)}$, the probability of finding such an $\alpha$ is 
$$P = e^{-u\ln(u)(1+o(1))} \geq \frac{1}{2^{\tilde{O}(n^{1/2})}}.$$
For any $A\geq 2$, the number of possible vectors $(x_i)$ is $\binom{|S|}{n^{b}}A^{n^{b}}$. Thus, size of the search space satisfies 
\begin{align*}
\binom{|S|}{n^{b}}A^{n^{b}} &\geq A^{n^{b}} |S|^{n^{b}} \left( 1 - \frac{n^{b}}{|S|}\right)^{n^{b}}\\
&\approx A^{n^{b}} \left( \frac{|S|}{n^{b}}\right)^{n^{b}} \frac{e^{n^{b}}}{\sqrt{2\pi n^{1/2}}} \left( 1 - \frac{n^{b}}{|S|}\right)^{n^{b}}\\
&\geq \left( \frac{|S|}{n^{b}}\right)^{n^{b}} e^{-\frac{n^{2b}}{|S|}}\\
&\geq \left( \frac{|S|}{n^{b}}\right)^{n^{b}} e^{-1}\ \text{ as } b\leq 1 \text{ and } |S| = n^{2 + o(1)}\\
&\gg 2^{\tilde{O}(n^{1/2})}\ \text{ as } b > 1/2 + \varepsilon.
\end{align*}
Therefore, the search space is sufficient to find our $\alpha$. The run time 
of this procedure is dominated by the reduction of $\mathcal{L}_q$ which takes $2^{O(k)} = 2^{O(n^{1/2 + \varepsilon})}$. 
\end{proof}

\begin{remark}
The constant $\varepsilon > 0$ can be replaced by a value of the form $\frac{\operatorname{Poly}(\log\log(n))}{\log(n)}$ where the numerator has a high enough degree 
to ensure that $n^{1/2 + \varepsilon} \gg \tilde{O}(n^{1/2})$. Then the complexity of one step of the $\qg$-descent is in $2^{n^{1/2 + o(1)}}$ which has 
to be repeated $\frac{1}{\varepsilon} \leq n$ times. 
\end{remark}

\paragraph{\textbf{Full procedure}}

Assume we are given an arbitrary ideal $I\subseteq\OK$ as input. Before initiating the $\qg$-descent, we need to 
break it down into a product of primes $\qg_i$ with $\log(\qg_i)\leq n$. To do this, 
we multiply $\ag$ by an random power product of ideals 
of the factor base $\mathcal{B}$, and BKZ-reduce it. This is done by finding 
a short element $\phi\in\cg$ where $\ag^{-1}=\frac{1}{d}\cg$ 
with $d\in\Z_{>0}$ and $\cg\subseteq \OO$. When using the BKZ algorithm with block size $l = n^{1/2}$, such a short element satisfies 
$$\| \phi \| \leq n^{n^{1/2}/4} |\Delta|^{\frac{1}{2n}} \Nm(\cg)^{\frac{1}{n}}.$$ 
Then, the ideal $\bg:= \frac{\phi}{l}\ag$ satisfies 
$\bg\subseteq\OK$, $\Nm(\bg)\leq 2^{\frac{n^{3/2}\log(n)}{2}} \sqrt{|\Delta|}$ and it is 
$2^n$-smooth with probability $\frac{1}{2^{\tilde{O}(n^{1/2}(1+o(1))}}$ by following the same argument as in the proof 
of Proposition~\ref{prop:run_time_qstep}. Repeating the process $2^{n^{1/2 + o(1)}}$ times yields the 
original $2^n$-smooth decomposition. 

\begin{algorithm}[ht]
\caption{Initial decomposition}
\begin{algorithmic}[1]\label{alg:LLL_red}
\REQUIRE An ideal $I\subseteq \OK$, factor base $\mathcal{B}$, and a constant $A\geq 1$.
\ENSURE $\alpha\in K$ such that $\bg:= (\alpha)I$ is $2^{n}$-smooth. 
\STATE $\ag\leftarrow \prod_i\pg_i^{e_i}$ for random $e_i\leq A$ (use a square-and-multiply strategy).
\STATE $\cg \leftarrow l\ag^{-1}$ where $l$ is the denominator of $\ag$.
\STATE Find a BKZ reduced $\phi\in\cg$ with block size $l = n^{1/2}$. 
\STATE $\alpha\leftarrow \frac{\phi}{l}$.
\STATE \textbf{if} $\bg:= (\alpha)I$ is not $2^{n}$-smooth \textbf{then} go to Step~1. 
\RETURN $\alpha$.
\end{algorithmic}
\end{algorithm}

\begin{algorithm}
\caption{$\q$-descent}
\begin{algorithmic}[1]\label{alg:full_q_descent}
\REQUIRE $I\subseteq \OK$ and an arbitrarily small $\varepsilon > 0$. 
\ENSURE Prime ideals $(\qg_i)_{i\leq l}\in\mathcal{B}$ with $\Nm(\q_i)\leq 2^{n^{1/2}}$, integers $(e_i)$ and $(\phi_j)_{j\leq k}\in K$ such that 
$I =\prod_{j\leq k}(\phi_j) \cdot\prod_{i\leq l}\mathfrak{q}_i^{e_i}$
\STATE Decompose $I = (\phi_1)\prod_i\qg_i$ using Algorithm~\ref{alg:LLL_red}.
\STATE $\mathrm{genList}\leftarrow \left\lbrace \phi_1\right\rbrace $, $\mathrm{primeList}\leftarrow \left\lbrace\mathfrak{q}_1,\hdots,\mathfrak{q}_l\right\rbrace$, $\mathrm{expList}\leftarrow \{1,\cdots,1\}$.
\STATE $b\leftarrow 1$. 
\WHILE {$b > 1/2$}
\FOR {$\q\in\mathrm{primeList}$  with $\log(\Nm(\qg))> n^{b-\varepsilon}$}
\STATE Find $(\mathfrak{q}_i)_{i\leq l}$, $(e_i)_{i\leq l}$ and $\phi_k$ such that 
$\mathfrak{q}=(\phi_k)\prod_{i\leq l}\mathfrak{q}_i^{e_i}$ by a round of the $\qg$-descent. 
\STATE $\mathrm{genList}\leftarrow \mathrm{genList}\cup\left\lbrace\phi_k\right\rbrace$, $\mathrm{primeList}\leftarrow\mathrm{primeList}\cup\left\lbrace \mathfrak{q}_1,\hdots,\mathfrak{q}_l\right\rbrace$.
\STATE $\mathrm{expList}\leftarrow \{e_1,\cdots,e_l\}$.
\STATE Remove $\qg$ from $\mathrm{primeList}$ and $\mathrm{expList}$.
\ENDFOR
\STATE $b\leftarrow b-\varepsilon$.
\ENDWHILE
\RETURN $\mathrm{genList}$, $\mathrm{primeList}$, $\mathrm{expList}$.
\end{algorithmic}
\end{algorithm}

\begin{proposition}
Algorithm~\ref{alg:full_q_descent} decomposes $I$ as an $2^{n^{1/2}}$-smooth product in $\Cl(\OK)$ 
in time $2^{O(n^{1/2+\varepsilon})}$ for arbitrarily small $\varepsilon > 0$. 
\end{proposition}

\begin{proof}
The cost of the initial decomposition is in $2^{O(l)}$ where $l = n^{1/2}$ is the block size of the BKZ reduction. 
The depth of the decomposition tree arising from the $\qg$-descent is $\frac{1}{2\varepsilon}$ which is 
a constant. The arity of this tree is at most $n^2$, so Algorithm~\ref{alg:one_round} is 
called at most polynomially many times. 
\end{proof}

\section{Practical improvements to the computation of $\Cl(\OK)$ and the resolution of the PIP}\label{sec:improvements}

The subexponential methods we presented in Section~\ref{sec:CL_group} and Section~\ref{sec:q_desc} have 
a heuristic asymptotic run time in $2^{O(n^{1/2+\varepsilon})}$, and depending on the choices made for 
implementation, the practical performances of these algorithms can vary significantly. In this section, we 
present some practical improvements which do not affect the asymptotic complexity, but which impact 
the practical run time. As the computation of the ideal class group and the resolution of the 
Principal Ideal Problem are widely studied problems, there are many existing improvements that readily 
apply to our algorithm, and many of them are folklore. For example, for each relation $\prod_i\pg_i = (\alpha)$ 
we find, we immediately get an additional $|\operatorname{Gal}(K/\Q)|-1$ others by including 
$\prod_i\pg_i^{\sigma} = (\sigma(\alpha))$ for all $\sigma \in  \operatorname{Gal}(K/\Q)$. This 
method is present in the class group algorithm available in both PARI~\cite{pari} and Magma~\cite{magma}. The large prime variants are 
another folklore practical improvement. It was originally described in the context of integer factorization~\cite{Lenstra:2LP,Lenstra:3LP}, 
and was successfully adapted to the resolution of the DLP in finite fields, and in the Jacobian of curves. 
The single large prime variant for computing $\Cl(\OK)$ was first presented by Jacobson~\cite{JacobsonPhd} while 
the double large prime variant was first successfully used for computing class groups by Biasse~\cite{biasse}. Other improvements 
impacting the practical performances of our methods include quadratic~\cite{Jdl} and lattice~\cite{biasse_fieker} 
sieving, as well as optimized methods of computing the HNF and the SNF of a large integer matrix. In this section, we 
restrict ourselves to improvements that are very specific to the settings of the algorithm. More specifically, 
we develop two points:
\begin{itemize}
 \item We show how to efficiently reduce the resolution of the PIP in $\Cl(\OO_{\Q(\zeta_N)})$ to an instance of 
 the PIP in $\Cl(\OO_{\Q(\zeta_N)^+})$.
 \item We show how to enhance the relation search by looking for elements of small norm near the cyclotomic units. 
 \item We show how to improve the basic step of the $\qg$-descent by enumerating many short elements in a given $\qg$. 
\end{itemize}

\paragraph{\textbf{Solving the PIP in the maximal real subfield}}

Using subfields (and in particular the maximal real subfield $K^+$) is folklore. Several references use a variant 
around this idea. We specifically rely on the Gentry-Szydlo method~\cite{Gentry_Szydlo} 
and its extension by Howgrave-Graham and Szydlo~\cite{szydlo_norm_eq} for solving norm 
equations of the form $\Nm_{\Q(\zeta_N)/\Q(\zeta_N)^+}(x) = g$. Halving the degree of the 
field can have a significant impact on the practical behavior of the PIP algorithm. 
The input size which is given by $\log(|\Delta|)$ is halved in the case of a power of 
two cyclotomic. In addition, the LLL algorithm has a practical behavior significantly 
better than the worst case estimates for lower dimensions. Indeed, it is expected to 
return a basis $(b_i)$ such that $\frac{\|b_1\|}{\det(L)^{1/d}}\approx (1.02)^d$ (see~\cite{LLL_average}). This 
makes a difference in practice during a relation search based on the search for small elements 
in LLL-reduced bases of ideals\footnote{In $K^+$ the height of the defining polynomial is higher, therefore other 
relation search methods may be used}. 

Given an input ideal $I \subseteq \OK$, we want to find a generator of the ideal $I'$ generated 
by $\Nm_{K/K^+}(I)$ in $\OO_{K^+}$. The ideal $I'$ is principal. 
However, the norm map is not surjective. This means that there can be a generator 
of $I'$ that is not the norm of a generator of $I$. If we solve the PIP for $I'$ in $\OO_{K^+}$ and 
find $g\in\OO_{K^+}$ such that $I' = (g)$, then 
we need to find another generator $g'$ that is the norm of an element in $K$. To proceed, 
we find the right unit $u\in\OO_{K^+}^*$ such that $g' = ug$ is totally positive. Given a 
set of fundamental units $(u_1,\cdots,u_r)$ for $\OO_{K^+}^*$, this can be done by solving 
a linear system in $\F_2$. For $\alpha\in K^+$, let $\delta(\alpha)\in\F_2^r$ be the signature of 
$x$, i.e. $\delta(\alpha)_i = 0$ if $\sigma_i(\alpha)\geq 0$, and $\delta(\alpha)_i = 1$ otherwise. Let 
$M = (\delta(u_i)_j)_{i,j\leq r}$, then the right product of the $u_i$ turning $g$ into a totally 
positive number is $u = \prod_i u_i^{x_i}$ where $(x_1,\cdots,x_r)M = \delta(g)$. 

\begin{proposition}
When $K$ is of the form $K = \Q(\zeta_{2^s})$, there is always a solution to the system  $(x_1,\cdots,x_r)M = \delta(g)$.
\end{proposition}

\begin{proof}
According to~\cite[Prop. 1]{signature_unit}, there are units of arbitrary signature in $\Q(\zeta_{2^s})$. Therefore, 
there must be a linear combination of the signatures of the units in the generating set that matches $\delta(g)$. 
\end{proof}

\begin{proposition}
When $K = \Q(\zeta_{2^s})$, a totally positive generator $g'$ of $I'$ is necessarily the norm of a generator of $I$. 
\end{proposition}

\begin{proof}
$I'$ is generated by at least one 
totally positive number (i.e. the image $\Nm_{K/K^+}(g_0)$ of a generator $g_0$ of $I$ by the norm map). 
Then from~\cite{squares_unit,weber}, 
we know that the totally positive units are exactly the squares of units, which are also the norms 
of the units of $\OK$ that are in $\OO_{K^+}$. The two totally positive generators $g',\Nm_{K/K^+}(g_0)$ of $I'$ differ by a 
totally positive unit, hence a square, and hence the image of a unit $u_0$ of $\OK\cap\OO_{K^+}$ by the norm map, 
i.e. $g' = \Nm_{K/K^+}(u_0)\Nm_{K/K^+}(g_0) = \Nm_{K/K^+}(u_0g_0)$, which is the image of a generator of 
$I$. 
\end{proof}

Before applying the Howgrave-Graham and Szydlo~\cite{szydlo_norm_eq} norm equation resolution method, 
we need to make sure that the input is polynomially bounded, which is not guaranteed if we take 
an arbitrary solution to the PIP in $\OO_{K^+}$ (even after adjusting the signature). However, we 
know the existence of a short (totally positive) generator of $I'$, namely the norm of a short generator 
of $I$. We are facing an instance of the Bounded Distance Decoding problem similar to the one solved 
by Cramer et al.~\cite{CVP_cyclo}. The logarithms of the norms of the cyclotomic units enjoy similar geometric properties 
as the logarithms of the cyclotomic units themselves. Therefore, applying the method of~\cite{CVP_cyclo} 
on $g'$ yields a small generator $g''$ of $I'$ on which we can apply the algorithm for the resolution of the 
norm equation $\Nm_{K/K^+}(x) = g''$ given in~\cite{szydlo_norm_eq}.

\begin{algorithm}[ht]
\caption{Reduction from the PIP in $\OK$ to the PIP in $\OO_{K^+}$}
\begin{algorithmic}[1]\label{alg:PIP_subfield}
\REQUIRE $I\subseteq\OK$, $I'=\Nm_{K/K^+}(I)\OO_{K^+}$, $ \OO_{K^+}^*= \langle u_1,\cdots,u_r\rangle$, $g$ with $I'=(g)$. 
\ENSURE $g_0\in\OK$ with $I = (g_0)$. 
\STATE $M\leftarrow (\delta(u_i)_j)_{i,j\leq r}$, $y\leftarrow \delta(g)$. 
\STATE Solve $(x_1,\cdots,x_r)M = y$. $g'\leftarrow \prod_i u_i^{x_i}$. 
\STATE Find a close vector $\Log(u)$ to $\Log(g')$ in $\mathcal{L} = \Z\Log(u'_1)+\cdots+\Z\Log(u'_r)$ where the $u'_i$ are the norms of the 
cyclotomic units of $\OK$ by using the methods of~\cite{CVP_cyclo}.
\STATE $g''\leftarrow g'\cdot u^{-1}$. 
\STATE Solve $\Nm_{K/K^+}(g_0) = g''$. 
\RETURN $g_0$.
\end{algorithmic}
\end{algorithm}

\paragraph{\textbf{Search for relations around the cyclotomic units}}

Our relation search method to compute $\Cl(\OK)$ consists of drawing polynomials $f(\zeta_N)$ 
with small coefficients at random and to check the norm of the resulting algebraic 
integer for smoothness. As we saw, the algebraic norm is bounded from above by a function of 
the length of the vector of coefficients of $f(X)$. Therefore, it is natural to search 
elements represented by a very small coefficient vector. 

We observe that $\Nm(f(\zeta_N))$ is a polynomial in the coefficients of $f$, and is therefore 
a continuous function. In cyclotomic fields, we know a set of minimums for this function before 
hand: the cyclotomic units $u_a = \frac{\zeta_N^a - 1}{\zeta_N -1}$. For each $a$, $\Nm(u_a) = 1$, 
and we expect small variations around the $u_a$ to yield algebraic numbers of small norm, 
although the coefficient vector of the corresponding polynomials might be long. In Table~\ref{tab:num_ev}, 
we compared the strategy consisting of drawing algebraic integers that are small variations around 
the cyclotomic units to the sampling of random coefficient vectors. We drew coefficient vectors 
of the same Hamming weight $w = 10,20,30,50,75,100$ in $K = \Q(\zeta_{2^{512}})$. 
For each $w$, we drew 100 random coefficient vectors 
of coefficients in $\{0,1\}$ (denoted by ``Random vectors''), and we drew 100 elements that differed 
from a cyclotomic unit $u$ of weight $w$ by a term in $\zeta_N^i$ for some $i$.

\begin{table}[ht]
\caption{Average $\log(\Nm(\alpha))$}
\label{tab:num_ev}
\begin{center}
 \begin{tabular}{|r|r|r|}
\hline
\multicolumn{1}{|r|}{Weight} & 
\multicolumn{1}{r|}{Random vectors}& 
\multicolumn{1}{|r|}{Unit variations}  \\
\hline
10 & 301 & 154 \\
20 & 430 & 156 \\
30 & 503 & 158 \\
50 & 586 & 159 \\
75 & 638 & 156 \\
100 & 674 & 154 \\
\hline
\end{tabular}
\end{center}
\end{table}

In Table~\ref{tab:num_ev}, we observe that the size of the elements $\alpha$ 
increases as the Hamming weight of their coefficient vector over the power 
basis gets larger. Meanwhile, at a comparable Hamming weight, the size of a small 
variation around a cyclotomic unit does not seem to be affected. In fact, it is 
the distance to the unit that seems to rule the size of $\alpha$. For example, we 
measured small variations around $u_{101}$ (of Hamming weight 101).
\begin{itemize}
 \item $\log(u_{101} + \zeta_N^3) = 151$.
 \item $\log(u_{101} + \zeta_N^3 -  \zeta_N^{150}) = 204$.
 \item $\log(u_{101} + \zeta_N^3 + \zeta_N^{60} - \zeta_N^{150}) = 238$.
 \item $\log(u_{101} + \zeta_N^3 + \zeta_N^{60} - \zeta_N^{150} + \zeta_N^{200}) = 283$.
\end{itemize}
If we draw $\alpha'$ and $\alpha$ at random such that $\log(|\Nm(\alpha')|)\leq \frac{|\log(\Nm(\alpha))|}{d}$ for some $d > 0$, 
then the expected time $T'$ to find a smooth $\alpha'$ satisfies $T'\leq \sqrt[d]{T}$ where $T$ is 
the expected time to find a smooth $\alpha$. Our numerical results indicate that the unit variation method 
should provide a significant speed-up over the random vectors method even when we allow a Hamming distance larger 
than one.

\paragraph{\textbf{Faster enumeration of short elements in $\qg$}} 

During the $\qg$-descent, the decomposition of an ideal $\qg$ of norm less than $2^{n^b}$ with respect to ideals 
of norm less than $2^{n^{b-\varepsilon}}$ is obtained by multiplying it by random products of primes $\pg_i$ with small 
norm and BKZ-reduce $I := \qg\prod_i\pg_i^{e_i}$. When $I$ is  $2^{n^{b-\varepsilon}}$-smooth, this yields a smooth decomposition 
of $\qg$. The BKZ algorithm has to be used every time a different $I$ is calculated. Each of these calls to BKZ is 
expensive (it has the same subexponential complexity as the rest of the algorithm). In practice, we prefer to 
enumerate short vectors directly in $\qg$. Indeed, the BKZ reduction usually produces a fair amount of short vectors $\vec{b}_1,\cdots,\vec{b}_k$ in 
$\qg$, and we can obtain many short elements by taking short linear combinations $\sum_ia_i\vec{b}_i$ where $|a_i|\leq A$ for a 
well-chosen $A$. 

\begin{proposition}[under Heuristic~\ref{heur:lambda_1}]\label{prop:short_elt_q}
Let $\frac{2}{5} < a \leq \frac{1}{2}$, and $n^a \leq k \leq n$. Then for any $a > a_1>0$ such that $n^{a_1}\ll k$, for a large enough $n$,  
the BKZ algorithm with block size $l$ returns $m = n^{a - a_1}$ independent vectors $\vec{b}_i\in \qg$ such that 
$$\|\vec{b}_i\| \leq A l^{\frac{k}{2l}} \Nm(\qg)^{1/k},$$
where $A = 2^{n^{a_1}}$. In particular, by taking $(a_i)_{i\leq m}\in [-A,A]^m$, we obtain $2^{n^a}$ vectors $\vec{b} = \sum_ia_i\vec{b}_i\in\qg$ such that 
$$\|\vec{b}\| \leq \sqrt{n}A l^{\frac{k}{2l}} \Nm(\qg)^{1/k}.$$
\end{proposition}

\noindent To prove Proposition~\ref{prop:short_elt_q} given $n^a \leq k \leq n$, we need to assume a heuristic on the length of the shortest vector 
of the lattice $\mathcal{L}_q\in\Z^n$ generated by the coefficients of the upper 
$k\times k$ corner of the HNF of the lattice of the coefficient of the elements of 
$\qg$ on the power basis. 
$$\mathcal{L}_q := \Z v_{1,1} + \Z (v_{2,2}\zeta_N + v_{2,1}) + \cdots + \Z(v_{k,k}\zeta_N^k + v_{k,k-1}\zeta_N^{k-1}+\cdots + v_k)\subseteq \qg,$$
where $N = q^s$ is the conductor of $K$. The coefficients $v_{i,j}$ are given 
by the HNF of the $\Z$-basis of $\qg$ which has the shape
\[ H:=\left( \begin{array}{ccccc}
v_{1,1}    &  0    & \hdots & \hdots &     0    \\
v_{2,1}    &  v_{2,2}    & \ddots &        &    \vdots \\
\vdots &  0    & \ddots & \ddots &    \vdots \\
\vdots & \vdots& \ddots &  \ddots&      0    \\
v_{N,1}& 0       & \hdots & 0      &   v_{N,N}\\
\end{array} \right).\] 
For Proposition~\ref{prop:short_elt_q}, we need to assume that the first minima $\lambda_1$ of $\mathcal{L}_q$ is at least 
$\det(\mathcal{L}_q)^{1/k}$. This is a much weaker assumption than the Gaussian heuristics which state that 
$\lambda_1\geq \sqrt{\frac{k}{2e\pi}}\det(\mathcal{L}_q)^{1/k}$. Moreover, we can actually prove this statement in the 
special case $k=n$ due to the algebraic structure of $\qg$. 

\begin{lemma}
Let $\lambda_1$ be the length of a shortest vector of $\mathcal{L}_q$ for $k=n$. Then $\lambda_1\geq \det(\mathcal{L}_q)^{1/n}$.
\end{lemma}

\begin{proof}
Let $(a_1,\cdots,a_n)$ be a shortest vector of $\mathcal{L}_q$ and $\alpha := \sum_i a_i\zeta_N^i$. On the one hand we have 
$\Nm(\alpha) \geq \Nm(I) \geq \Nm(\mathcal{L}_q)$. From the inequality between the arithmetic and the geometric mean, 
we also get $\|\alpha\| \geq n^{n/2}|\Nm(\alpha)|$, which means that 
$\|\alpha\| \geq \sqrt{n}\det(\mathcal{L}_q)^{1/n}$. From~\cite[Sec. 3.1]{HNF_symb}, we know that 
$\|\alpha\| \leq C_1 \|(a_1,\cdots,a_n)\|_\infty$ for $C_1 := \max_i\|\zeta_N^i\| = \sqrt{n}$ (where $\|x\|_\infty$ denotes the 
infinity norm of $x$). Since 
$\|(a_1,\cdots,a_n)\|_\infty\leq \|(a_1,\cdots,a_n)\|$, we obtain
$$\sqrt{n}\det(\mathcal{L}_q)^{1/n}\leq \|\alpha\|\leq C_1 \|(a_1,\cdots,a_n)\|_\infty  \leq  \sqrt{n}  \|(a_1,\cdots,a_n)\| = \sqrt{n}\lambda_1.$$

\end{proof}

\begin{heuristic}\label{heur:lambda_1}
Let $\lambda_1$ be the length of a shortest vector of $\mathcal{L}_q$. Then $\lambda_1\geq \det(\mathcal{L}_q)^{1/k}$.
\end{heuristic}

\begin{proof}[of Proposition~\ref{prop:short_elt_q}]
Let $\lambda_i$ be the $i$-th minima of the lattice $\mathcal{L}_q\in\Z^n$ generated by the coefficients of the upper 
$k\times k$ corner of the HNF of the lattice of the coefficient of the elements of 
$\qg$ on the power basis. By Minkowski's second theorem, we have 
$$\lambda_1\cdots\lambda_k \vol(K) \leq 2^k \det(\mathcal{L}_q),$$
where $K$ is the unit sphere in $\R^k$. Its volume satisfies $\vol(K) = \frac{\pi^{k/2}}{\Gamma(1 + k/2)}$, which yields
\begin{align*}
\lambda_1\cdots\lambda_k &\leq 2^k \frac{\Gamma(1 + k/2)}{\pi^{k/2}}\det(\mathcal{L}_q)\\
&\sim 2^k \sqrt{2\pi k/2}\left(\frac{k}{2\pi e}\right)^{k/2}\det(\mathcal{L}_q).
\end{align*}
We assume that $\lambda_i \geq \det(\mathcal{L}_q)^{1/k}$. Assume by way of contradiction that 
there are more than $k-m$ minimas $\lambda_i$  such that 
$\lambda_i > A \det(\mathcal{L}_q)^{1/k}$. Then since the other vectors have length at least $\det(\mathcal{L}_q)^{1/k}$, then 
their product satisfies 
$$\lambda_1\cdots\lambda_k > A^{k-m}\det(\mathcal{L}_q).$$
But given the definition of $A$, we have 
$$A^{k-m} = 2^{n^{a_1}k(1 + o(1))} \gg 2^k \sqrt{2\pi k/2}\left(\frac{k}{2\pi e}\right)^{k/2}.$$
Therefore we must have at least $m$ minimas satisfying $\lambda_i\leq A \det(\mathcal{L}_q)^{1/k}$, and a BKZ-reduced basis $(\vec{b}_i)_{i\leq k}$ 
must have its first $m$ vectors satisfying $\|\vec{b}_i\|\leq A l^{\frac{k}{2l}}\det(\mathcal{L}_q)^{1/k}$
\end{proof}

\section{PIP and $\gamma$-SVP in $\Q(\zeta_{p^s})$ with precomputation}\label{sec:short_PIP}

In Section~\ref{sec:q_desc}, we showed how to compute the PIP in heuristic subexponential time $2^{O(n^{1/2+ \varepsilon})}$. 
This provides an attack against schemes relying on the hardness of finding a short generator of a principal ideal 
such as~\cite{fre_smart,mult_maps}. Also, according to~\cite{CVP_cyclo}, the size of a short generator of $I$ resulting from the BDD algorithm 
on the log-unit lattice is within $e^{O(\sqrt{n})}$ of the first minima of the ideal lattice $I$. Moreover, it was recently conjectured~\cite{Stickelberger} that 
for most fields $\Q(\zeta)$, any ideal of $\OO_{\Q(\zeta)}$ was within a short enough ideal multiple from a principal ideal. Therefore, 
solutions to the PIP in $\OO_{\Q(\zeta)}$ yields solution to $\gamma$-SVP in ideals of  $\OO_{\Q(\zeta)}$ for $\gamma \in e^{O(\sqrt{n})}$. 
Since the methods of Section~\ref{sec:q_desc} run in heuristic subexponential time $2^{O(n^{1/2+ \varepsilon})}$, this does not 
offer a better trade-off than using a BKZ reduction with block size in $O(\sqrt{n})$. 

In this section, we show how to leverage a subexponential precomputation to find a short generator 
of ideals of  $\OO_{\Q(\zeta_N)}$ for 
$N = p^s$ in time better than $2^{O(\sqrt{n})}$. As these short generators are solutions to $\gamma$-SVP for 
$\gamma \in e^{O(\sqrt{n})}$ , we thus achieve a better trade-off than the BKZ reduction. To the 
best of our knowledge, this is the first time a classica method for solving $\gamma$-SVP beats the BKZ time/approximation trade-off even with a precomputation on the 
field. The general idea is to use the ideal class group computation given by Algorithm~\ref{alg:class_group} with a larger factor base bound $B$. 
This gives us a basis for all relations between the prime ideals of norm less than $B$. Then given an input ideal $I$, we first compute 
$\alpha$ and $\qg_1,\cdots,\qg_k$ such that $I = (\alpha)\prod_i\qg_i$ with $\Nm(\qg_i)\leq B$, and then we solve the PIP problem for $\prod_i\qg_i$. 
The larger the smoothness bound $B$ is, the more expensive the precomputation gets. Meanwhile all subsequent PIP resolutions using this precomputation 
get faster since their 
run time is dominated by the decomposition in $\Cl(\OK)$ of $I$ as a product of ideals of norm bounded by $B$. 
\begin{enumerate}
 \item Precomputation: Given $B$, find a basis for the lattice $\mathcal{L}$ of vectors $(e_i)_{i\leq k}$ such that $\prod_i \qg_i^{e_i}\sim (1)$ for the $\Nm(\qg_i)\leq B$. 
 \item Decomposition: Given a principal idea; $I$ and $B$, find $(\qg_i)_{i\leq k}$ and $\alpha\in K$ such that $I = (\alpha)\prod_i \qg_i^{e_i}$. 
 \item Resolution: Given the decomposition of $I$ and $\mathcal{L}$, find a solution of $\gamma$-SVP for $I$. 
\end{enumerate}

\paragraph{\textbf{Precomputation step}}

Given the soothness bound $B$, Algorithm~\ref{alg:class_group} returns a generating set $(\vec{b}_1,\cdots,\vec{b}_t)$ of the lattice $\mathcal{L}$ of 
vectors $\vec{x}$ such that $\mathcal{B}^{\vec{x}}\sim (1)$ for $\mathcal{B} = \{ \pg\ \mid \ \Nm(\pg)\leq B\}$ 
together with $\alpha_i\in K$ such that $\mathcal{B}^{\vec{b}_i} = (\alpha_i)$. Our precomputation step additionally 
processes this basis and the $(\alpha_i)$ to return an HNF-reduced basis $(\vec{h}_1,\cdots,\vec{h}_m)$ for $\mathcal{L}$. 
Using~\cite[Prop. 6.3]{Arne}, we can find a unimodular $U\in\operatorname{Gl}_{t\times t}(\Z)$ such that 
\[ UB 
  =\left( \begin{array}{cccc}
h_{11}& 0      & \hdots & 0      \\
\vdots & h_{22}& \ddots & \vdots \\
\vdots & \vdots & \ddots & 0      \\
*      & *      & \hdots & h_{m,m}\\
 0   & \hdots & \hdots & 0 \\
 \vdots &  &  &  \vdots \\
 0 & \hdots & \hdots & 0
 \end{array}
   \right)
\]
is the HNF of $B = (\vec{b}_i)_{i\leq k}$ with $\|U\|\leq \left( \sqrt{m}\|B\|\right)^m$ in time 
$$O\left( tm^{\theta - 1}\log(\beta) + tm\log(m)\operatorname{Mult}(\log(\beta))\right)$$ 
for $\beta =  \left( \sqrt{m}\|B\|\right)^m$, $\operatorname{Mult}(x)$ the complexity of $x$-bit 
integer multiplication, and $2\leq \theta\leq 3$ the matrix multiplication 
exponent. The matrix $H = UB$ has a small essential part. Under GRH, $h_{i,i} = 1$ for $i> 12\log(|\Delta|)^2$. 
We leverage this to facilitate the resolution of the linear system giving the solution to the PIP. 
However, for this to yield a generator (as opposed to just the answer whether or not $I$ is principal), 
we need to compute the $(\beta_i)_{i\leq m}$ such that $\prod_{j\leq t}\alpha_j^{U_{i,j}} = \beta_i$ 
for $i\leq m$. 
As the coefficients of $U$ and the number of terms $m$ in the product are large, we cannot afford to 
write down these algebraic numbers on the power basis. However, we know that they are used to compute 
an element of $I$ whose length is within $e^{\tilde{O}(\sqrt{n})}$ of the first minima $\lambda_1(I)\leq \sqrt{n}|\Delta|^{1/n}\Nm(I)^{1/n}$ 
of the ideal lattice $I$. So we compute $\beta_i\bmod \pg_j$ for a collection of split prime ideals 
$\Pg_1,\cdots,\Pg_s$ such that $\prod_j\Nm(\Pg_j) \leq e^n n^{n/2} |\Delta|\Nm(I)$. 

We also need to keep $\Log(\beta_i)$ as part of the precomputation as they are needed for computing 
a short generator of a principal ideal by solving an instance of the BDD in $\Log(\Z[\zeta_N]^*)$. Each 
of these values satisfies $\Log(\beta_i) = \sum_{j\leq k} U_{i,j}\Log(\alpha_j)$. The logarithm vectors of 
the $\alpha_j$ have polynomial size, but the bit size of the $U_{i,j}$ is in $2^{O(n^\kappa)}$ for $1/2 < \kappa < 1$ 
depending on the parameters of the precomputation. As we are aiming at lowering down the 
cost of subsequent resolutions of the short-PIP which requires the values of $\Log(\beta_i)$, we must 
find different generators corresponding to the rows of $UB$. Let $i_0$ such that 
$\{\pg_1,\cdots,\pg_{i_0}\} = \{\pg\ \mid \ \Nm(\pg) < 12\ln(|\Delta|)^2\}$ and $i_1$ such that 
$\Nm(\pg_{i_1}) > 12\ln(|\Delta|)^2$, 
then $\beta_{i_1}$ is the generator of the principal ideal $\pg_{i_1}\cdot\left( \prod_{j\leq i_o}\pg_j^{h_{i_0,j}}\right) = (\beta_{i_1})\OK$. 
We use the BBD solution in the log-unit lattice of Cramer et al.~\cite{CVP_cyclo} to find $\beta'_{i_1}$ such that 
$\pg_{i_1}\cdot\left( \prod_{j\leq i_o}\pg_j^{h_{i_0,j}}\right) = (\beta'_{i_1})\OK$ and such that the bit size of the representation of 
$\Log(\beta'_{i_1})$ is bounded.

\begin{algorithm}[ht]
\caption{Precomputation step}
\begin{algorithmic}[1]\label{alg:precomputation}
\REQUIRE Split primes $(\Pg_i)_{i\leq s}$, $B > 0$, and conductor $N$. 
\ENSURE $H$ in HNF form with $(\beta_j\bmod \pg_i)$ such that $\mathcal{B}^{H_i} = (\beta_i)\OK$ for $\mathcal{B} = \{ \pg\ \mid \ \Nm(\pg)\leq B\}$, 
and $(\Log(\beta_j))_{j\leq m}$. 
\STATE Compute $\mathcal{B} = \{ \pg\ \mid \ \Nm(\pg)\leq B\}$. Let $m := |\mathcal{B}|$.
\STATE Compute a generating set $\vec{b}_1,\cdots,\vec{b}_t$ of the lattice $\mathcal{L}\subseteq\Z^m$ of vectors $\vec{x}$ such that $\mathcal{B}^{\vec{x}}\sim (1)$ 
and $(\alpha_i)_{i\leq t}$ such that $\mathcal{B}^{\vec{b}_i} = (\alpha_i)$ using Algorithm~\ref{alg:class_group}. 
\STATE Find $U\in \operatorname{Gl}_{t\times t}(\Z)$ such that $U\cdot B = H$ in HNF form for $B = (\vec{b}_i)_{i\leq t}$ using~\cite[Prop. 6.3]{Arne}. 
\STATE \textbf{for} $i\leq t$, compute $\Log(\beta_i) = \sum_j u_{i,j}\Log(\alpha_j)$. 
\STATE \textbf{for} $i\leq t$, $j\leq s$ \textbf{do} Compute $\beta_i \bmod \Pg_j:=\alpha_1^{u_{i,1}}\cdots\alpha_k^{u_{i,k}}\bmod \Pg_j$. 
\STATE Compute  the cyclotomic units $u_1,\cdots,u_r$ of $\Q(\zeta_N)$ . 
\FOR {$1\leq i\leq m$}
\STATE $I_i\leftarrow \prod_{j\leq i_0}\pg_j^{h_{i,j}}$.
\STATE Let $(x_j)_{j\leq r}$ be the output of Algorithm~\ref{alg:CVP} solving the BDD on input $I_i,\Log(\beta_i)$.
\STATE $\beta_j\bmod\Pg_k\leftarrow u_1^{x_1}\cdots u_{r}^{x_r}\beta_j\bmod \Pg_k$ for $j\leq m$, $k\leq s$. 
\STATE $\Log(\beta_j)\leftarrow x_1\Log(u_1)+\cdots+ x_r\Log(u_{r}) + \Log(\beta_j)$ for $j\leq m$.
\ENDFOR
\RETURN $H , (\beta_i\bmod \Pg_j)_{i\leq m,j\leq s}, (\Log(\beta_j))_{j\leq m}$
\end{algorithmic}
\end{algorithm}

\begin{proposition}[GRH + Heuristic~\ref{heuristic:smoothness_ideals}]
Assume that $B \in 2^{O(n^\kappa)} $ for $\kappa \geq 1/2 $, and that $s\in O ( \log(|\Delta|))$, 
$\Nm(\pg_j)\in O(\log(|\Delta|)^2)$, then the heuristic expected run time of 
Algorithm~\ref{alg:precomputation} is in $2^{O(n^\kappa)}$, 
and the bit size of the representation of the $\Log(\beta_j)$ is polynomial 
in $n$. 
\end{proposition}

\begin{proof}
The run time of Algorithm~\ref{alg:precomputation} is clearly dominated by the cost of the search for 
relations and the computation of the HNF of the relation matrix (together with the premultipliers). 
We need to bound the $\Log(\beta_i)$. 
 The upper bound on a generator $\beta_i$ of the principal 
 ideal $I_i = \prod_{j\leq m}\pg_j^{h_{i,j}}$ is essentially given by the 
norm of the integral ideal $I_i$. When $i\leq i_0$, $I_i$ has the shape $I_i = \prod_{j\leq i_0}\pg_j^{h_{i_0,j}}$ 
while when $i > i_0$, $I_i$ is of the form $I_{i} := \pg_{i}\cdot\left( \prod_{j\leq i_0}\pg_j^{h_{i_0,j}}\right)$. 
For each $j\leq i_0$, $\Nm(\pg_j)\leq 12\ln(|\Delta|)^2$ while $\Nm(\pg_i)\in 2^{O(n^\kappa)}$ if $i > i_0$ and 
$h_{i,j}\leq |\Cl(\OK)|\in O(\sqrt{|\Delta|})$ for $i,j\leq i_0$. Therefore, in any case $\Nm(I_i)\in 2^{\tilde{O}(|\Delta|)}$, 
and $\|\beta_i\|\leq 2^{\tilde{O}(n^{1/2})}\Nm(I_i)^{1/n}\in 2^{\tilde{O}(|\Delta|)}$. For each $\sigma\in\operatorname{Gal}(K/\Q)$, 
$\max_\sigma|\sigma(\beta_i)|\leq \|\beta_i\|\in 2^{\tilde{O}(|\Delta|)}$, and 
$\min_\sigma|\sigma(\beta_i)|\geq \frac{|\Nm(\beta_i)|}{\left(\max_\sigma|\sigma(\beta_i)|\right)^{n-1}} \geq \frac{1}{2^{\tilde{O}(n|\Delta|)}}$. 
Therefore, for all $\sigma\in \operatorname{Gal}(K/\Q)$, $|\ln(|\sigma(\beta_i))|\in \tilde{O}(|\Delta|)$, and the representation of 
the vector $\Log(\beta_i)$ has a polynomial bit size in $n$. 
\end{proof}

\paragraph{\textbf{Decomposing $I$}}

The second step of our search for a generator of a principal ideal $I$ satisfying $\Nm(I)\leq 2^{n^b}$ is to break it down as a product of ideals of 
norm less than $ B \in 2^{O(n^{2-3a+\varepsilon})}$ for an arbitrarily small $\varepsilon>0$ in time $2^{O(n^{a+o(1)})}$ where 
\begin{enumerate}
 \item $2-3a + 2\varepsilon \leq b\leq 7a - 2$.
 \item $\frac{2}{5} < a < \frac{1}{2}$. 
\end{enumerate} 
We proceed by a $\qg$-descent procedure similar to that of Section~\ref{sec:q_desc}. Given a 
prime ideal $\qg$ such that $\log(\Nm(\qg)) \leq n^b$, 
we look for $\alpha\in\qg$ such that 
$(\alpha)/\qg = \prod\qg_i$ where the $\qg_i$ are prime ideals satisfying $\log(\Nm(\qg_i))\leq n^{b - \varepsilon}$. 
We look for small vectors of $\mathcal{L}_I$ with the same definition of $\mathcal{L}_I$ as in Section~\ref{sec:q_desc} 
by using the BKZ reduction method with block size $l := n^{a}$ and 
by parameterizing $\mathcal{L}_q$ with the degree $k := n^{\min\left\{4a-1,b+2a-1-\varepsilon\right\}}$. 

\begin{lemma}\label{lemma:smooth_BKZ}
Let $l\leq k\leq n$. By using a BKZ reduction on $\mathcal{L}_I$, we can find a vector  
in $\mathcal{L}_I$ of length less than $l^{k/2l}\Nm(I)^{\frac{1}{k}}$ in time $2^{O(l)}$.  
\end{lemma}

\begin{proof}
The determinant of $\mathcal{L}_I$ is that of the upper left $k\times k$ 
submatrix of $H$ and satisfies $\det(\mathcal{L}_I) \leq \prod_{i\leq N}v_{i,i} = \Nm(I)$. 
An BKZ-reduction with block length $l$ returns a basis 
whose first vector has length less than $l^{k/2l}(\det(\mathcal{L}_I))^{1/k}$ in time $2^{O(l)}$. 
\end{proof}

\begin{lemma}\label{lemma:smooth_combinations_BKZ}
In time $2^{O(l)}$, we can find an element $\alpha\in I$ 
such that $\Nm(\alpha)\leq l^{\frac{kn}{2l}(1+o(1))}\Nm(I)^{\frac{n}{k}}.$
\end{lemma}

\begin{proof}
Let $\alpha$ be the first vector of an BKZ-reduced basis of $\mathcal{L}_I$ with block size $k$. The calculation of this basis 
takes time $2^{O(k)}$ and by Lemma~\ref{lemma:smooth}, the length of its first vector $(\alpha_1,\cdots,\alpha_k)$ is bounded 
by $l^{k/2l}\Nm(I)^{\frac{1}{k}}$. 
By the same argument as in the proof of Proposition~\ref{prop:class_group}, 
the algebraic norm of $\alpha := \sum_i \alpha_i \zeta_N^i$ satisfies 
$$\Nm(\alpha) \leq \left.\sqrt{n}\right.^n \left( \|(\alpha_1,\cdots,\alpha_k)\| \right)^n
\leq \underbrace{n^{n/2} l^{kn/2l}}_{l^{\frac{kn}{2l}(1+o(1))}}\Nm(I)^{\frac{n}{k}}.$$
\end{proof}

\begin{algorithm}[ht]
\caption{One round of the $\qg$-descent}
\begin{algorithmic}[1]\label{alg:one_round_2}
\REQUIRE Prime ideal $\qg$ with $\log(\Nm(\qg))\leq n^b$, $\varepsilon > 0$, $a > 0$, $A\geq 1$, and a conductor $N$. 
\ENSURE Prime ideals $\pg_i$, integers $e_i$ and $\alpha\in\qg$ such that $(\alpha)/\qg = \prod_i\pg_i^{e_i}$ 
and $\log(\Nm(\pg_i))\leq n^{b-\varepsilon}$. 
\STATE $S\leftarrow \left\{\pg_i\ \text{such that }\Nm(\pg_i)\leq 12\log(|\Delta|)^2\right\}$.
\WHILE{no relation has been found}
\STATE $(x_i)\xleftarrow[]{R}[0,A]^{|S|}$ where $R$ is uniform over vectors of weight $n^{1/2}$. 
\STATE  $I\leftarrow \qg\prod_i \pg_i^{x_i}$.
\STATE Construct a basis for the lattice $\mathcal{L}_I$ with $k := n^{\min\left\{4a-1,b+2a-1-\varepsilon\right\}}$. 
\STATE BKZ-reduce $\mathcal{L}_I$ with block size $l :=n^{a}$ and let $\alpha$ be its first vector.
\STATE \textbf{if} $(\alpha)$ is $2^{n^{b-\varepsilon}}$-smooth, find $(\pg_i)$,$(e_i)$ such that $(\alpha) = \prod_i\pg_i^{e_i}$.  
\ENDWHILE
\RETURN $\alpha$, $(e_i)$, $(\pg_i)$, $(e_i)$. 	
\end{algorithmic}
\end{algorithm}

\begin{proposition}[GRH + Heuristic~\ref{heuristic:smoothness_ideals}]
Let $\varepsilon > 0$ be an arbitrarly small constant and let $b,c > 0$ be constants satisfying 
\begin{enumerate}
 \item $2-3a + 2\varepsilon \leq b \leq 7a - 2$.
 \item $\frac{2}{5} + \frac{\varepsilon}{5} \leq  a \leq \frac{1}{2}$. 
\end{enumerate} 
Let $\qg$ be a prime with $\log(\Nm(\qg))\leq n^b$. When $N = p^s$, 
there is a large enough constant $A$ such that Algorithm~\ref{alg:one_round_2} returns a decomposition of 
$\qg$ in $\Cl(\OK)$ as a product of primes $\pg_i$ with 
$\log(\Nm(\pg))\leq n^{b-\varepsilon}$ in time $2^{\tilde{O}(n^{a +  o(1) })}$. 
\end{proposition}

\begin{proof}
According to Lemma~\ref{lemma:smooth_combinations}, any $\alpha$ derived in Step~4 of Algorithm~\ref{alg:one_round_2} 
satisfies $\log(\Nm(\alpha)) \in O\left( \frac{nk\log(l)}{l}(1+o(1)) + \frac{n}{k}\log(\Nm(I))\right)$. 
As $k\leq n^{4a-1}$ and $l = n^a$, we get 
$$\frac{nk}{l}\leq n^{1+4c-1-c} n^{3a}.$$
The extra term multiplied to $\qg$ has norm $\Nm(\prod_i\pg_i^{x_i})\leq 2^{\tilde{O}(n^{1/2})}$, 
and since $b > 2-3a > \frac{1}{2}$, we have  
$\log(\Nm(I)) = n^{b}(1 + o(1))$. 
Moreover,  since $k = \min\{ n^{4a - 1} , n^{b+2a-1-\varepsilon}\}$, we get 
\begin{align*}
\frac{n}{k}\log(\Nm(I)) &\leq n^{1 + b - 4a + 1} (1+o(1)) \leq n^{3a}(1+o(1))\\
\frac{n}{k}\log(\Nm(I)) &\leq n^{1+b-(b+2a-1-\varepsilon)} (1+o(1)) \leq n^{2(1-a) + \varepsilon}(1+o(1)) \leq n^{3a}(1+o(1))
\end{align*}
The latter inequality follows from the fact that by definition $a\geq \frac{2}{5} + \frac{\varepsilon}{5}$. 
In each case, $\log(\Nm(\alpha)) \in O (n^{3a}(1+o(1)))$, and testing the smoothness of $\Nm(\alpha)$ with the 
Number Field Sieve takes time $2^{O(n^{a+o(1)})}$. As $k\leq n^{b+2a-1-\varepsilon}$, we also have 
$$\frac{nk}{l}\leq n^{1+b+2a-1-\varepsilon -a} = n^{a +b-\varepsilon}.$$
In addition, we can show that $k\geq n^{1-a+\varepsilon}$. Indeed, from the definition of $a,b$ we get
\begin{align*}
1-a + \varepsilon &\leq 4a-1  \Leftrightarrow \frac{2}{5} + \frac{\varepsilon}{5} \leq a\\
1-a + \varepsilon &\leq b + 2a -1 -\varepsilon \Leftrightarrow 2-3a + 2\varepsilon \leq b
\end{align*}
Therefore, we have the following inequality
$$\frac{n}{k}\log(\Nm(I)) \leq n^{1 + b -(1-a + \varepsilon)}(1+o(1)) = n^{a + b - \varepsilon}(1+o(1)).$$
This means that $\log(\Nm(\alpha)) \in O (n^{a + b - \varepsilon}(1+o(1)))$, and from 
Heuristic~\ref{heuristic:smoothness_ideals}, the number of $\alpha$ we need to test 
before obtaining one such that $(\alpha)/I$ is $2^{n^{b-\varepsilon}}$-smooth is bounded by 
$2^{O(n^{a + o(1)})}$. Following the same strategy as in the proof of Proposition~\ref{prop:run_time_qstep}, 
we can prove that the search space is asymptotically at least $2^{\tilde{O}(n^{1/2})} \gg 2^{O(n^{a + o(1)})}$ 
thus guaranteeing that it is large enough. 
For correctness, we also check that we always have $l\leq k\leq n$. 
First, $k\geq n^{1-a + \varepsilon}$, and 
$$1-a + \varepsilon \geq a \Leftrightarrow a \leq \frac{1}{2} + \frac{\varepsilon}{2}$$
\end{proof}
Since $a < \frac{1}{2}$, we must have $k \geq n^a = l$. On the other hand, we have 
$k = \min\{ n^{4a - 1} , n^{b+2a-1-\varepsilon}\}$, and $4a - 1 \leq 1 \Leftrightarrow a\leq \frac{1}{2}$, 
which is satisfied by definition of $a$. Therefore, we always have $k\leq n$.

\begin{algorithm}
\caption{$\q$-descent}
\begin{algorithmic}[1]\label{alg:full_q_descent_2}
\REQUIRE $I\subseteq \OK$ with $\Nm(I)\leq 2^{n^{b_0}}$, $\varepsilon > 0$ and $a$ such that 
$b_0\leq 7a - 2$ and $\frac{2}{5}+ \frac{\varepsilon}{5}\leq a \leq \frac{1}{2}$. 
\ENSURE Prime ideals $(\qg_i)_{i\leq l}\in\mathcal{B}$ with $\Nm(\q_i)\leq 2^{n^{2-3a+2\varepsilon}}$, integers $(e_i)$ and $(\phi_j)_{j\leq k}\in K$ such that 
$I =\prod_{j\leq k}(\phi_j) \cdot\prod_{i\leq l}\mathfrak{q}_i^{e_i}$.
\STATE $\mathrm{genList}\leftarrow \left\lbrace 1\right\rbrace $, $\mathrm{primeList}\leftarrow \left\lbrace I \right\rbrace$, 
$\mathrm{expList}\leftarrow \{1\}$.
\STATE $b\leftarrow b_0$.
\WHILE{$b > 2-3a+2\varepsilon$}
\FOR {$\q\in\mathrm{primeList}$ with $\Nm(\q)> n^{b-\varepsilon/2}$}
\STATE Find $(\mathfrak{q}_i)_{i\leq l}$, $(e_i)_{i\leq l}$ and $\phi_k$ such that 
$\mathfrak{q}=(\phi_k)\prod_{i\leq l}\mathfrak{q}_i^{e_i}$ by Algorithm~\ref{alg:one_round_2}.
\STATE $\mathrm{genList}\leftarrow \mathrm{genList}\cup\left\lbrace\phi_k\right\rbrace$, $\mathrm{primeList}\leftarrow\mathrm{primeList}\cup\left\lbrace \mathfrak{q}_1,\hdots,\mathfrak{q}_l\right\rbrace$.
\STATE $\mathrm{expList}\leftarrow \mathrm{expList}\cup\{e_1,\cdots,e_l\}$.
\STATE Remove $\q$ from $\mathrm{primeList},\mathrm{expList}$.
\ENDFOR
\STATE $b\leftarrow b - \varepsilon$. 
\ENDWHILE
\RETURN $\mathrm{genList}$, $\mathrm{primeList}$, $\mathrm{expList}$.
\end{algorithmic}
\end{algorithm}

\begin{proposition}[GRH + Heuristic~\ref{heuristic:smoothness_ideals}]
Algorithm~\ref{alg:full_q_descent} decomposes $I$ with $\Nm(I)\leq 2^{n^b}$ as an $2^{O(n^{2-3a+2\varepsilon})}$-smooth product in $\Cl(\OK)$ 
in time $2^{O\left((n^{a + o(1)}\right)}$ for arbitrarily small $\varepsilon > 0$ and 
\begin{enumerate}
 \item $2-3a + 2\varepsilon \leq b \leq 7a - 2$.
 \item $\frac{2}{5} + \frac{\varepsilon}{5} \leq  a \leq \frac{1}{2}$. 
\end{enumerate} 
\end{proposition}

\paragraph{\textbf{Resolution step}}

The $\qg$-descent gives us a decomposition of an input principal ideal $I$ of the form 
$$I = (\phi)\pg_1\cdots\pg_m$$
where $\phi\in K$ and $\Nm(\pg_i)\leq 2^{O(n^{2-3a+2\varepsilon})}$. We refine this decomposition 
into one that involves only primes of norm less than $12\ln(|\Delta|)^2$ by using the precomputed relation 
matrix and we solve a linear system to obtain a generator of $I$. This generator is then used to 
derive a short generator of $I$ by using the techniques of~\cite{CVP_cyclo}. 
We use the fact that under the GRH, the essential part of the HNF of the relation matrix 
precomputed has less than $12\ln(|\Delta|)^2$ columns (corresponding to the prime ideals of norm less than 
$12\ln(|\Delta|)^2$). The HNF has the form $H = \left(\begin{smallmatrix} H_1 & (0) \\ H_2 & I \end{smallmatrix}\right)$ 
where $I$ is an identity block corresponding to the relations of the form $\mathfrak{Q}\sim \prod_i\pg_i^{e_i}$ where $(-e_i)$ 
is a row vector of $H_2$, $\Nm(\pg_i)\leq 12\ln(|\Delta|)^2$ and $\mathfrak{Q}\in\mathcal{B}$, $\Nm(\mathfrak{Q}) > 12\ln(|\Delta|)^2$. 
Given a decomposition of an input ideal over $\mathcal{B}$, it is straightforward to rewrite all large prime ideals as products of 
the ideals of norm less than $12\ln(|\Delta|)^2$. We describe this procedure in Algorithm~\ref{alg:decomp}. 

\begin{algorithm}[ht]
\caption{Decomposition over a small generating set}
\begin{algorithmic}[1]\label{alg:decomp}
\REQUIRE Hermite form $H = \left(\begin{smallmatrix} H_1 & (0) \\ H_2 & I \end{smallmatrix}\right)$ of the 
matrix of relations between primes of $\mathcal{B}=\{\pg\ \mid\ \Nm(\pg)\leq 2^{n^\kappa}\}$ for some $\kappa > 0$, 
and input ideal $I$.
\ENSURE $(e_i)$ such that $I\sim \prod\pg_i^{e_i}$ for the $\pg_i$ such that $\Nm(\pg_i)\leq 12\ln(|\Delta|)^2$.
\STATE Use the $\qg$-descent to find $I\sim \prod_i\qg_i$ with $\qg_i\in\mathcal{B}$.
\STATE $i_0\leftarrow \dim(H_1)$. 
\FOR {$\qg_j\mid I$ with $\Nm(\qg_j) > 12\ln(|\Delta|)^2$}
\STATE Use the corresponding row in $H_2$ to find $\qg_j\sim \prod_{i\leq i_0}\pg_j^{f_i}$.
\STATE Update the decomposition of $I$.
\ENDFOR
\RETURN $(e_i)$ where $I\sim \prod_{i\leq i_0}\pg_i^{e_i}$. 
\end{algorithmic}
\end{algorithm}

\begin{algorithm}[ht]
\caption{Short generator with precomputation}
\begin{algorithmic}[1]\label{alg:gamma_SVP}
\REQUIRE Hermite form $H = \left(\begin{smallmatrix} H_1 & (0) \\ H_2 & I \end{smallmatrix}\right)$ of the 
relation matrix with $(\alpha_i\bmod \pg_j)_{i\leq k,j\leq l}$, $(\Log(\alpha_i))_{i\leq m}$ 
where $m = |\mathcal{B}|$ for $\mathcal{B} = \{ \pg\ \mid \ \Nm(\pg)\leq B\}$ for 
some $B>0$ such that $\mathcal{B}^{H_i} = (\alpha_i)\OK$, and input ideal $I$.
\ENSURE $\beta\in I$ with $\|\beta\|\leq e^{\tilde{O}(\sqrt{n})}\Nm(I)^{1/n}$. 
\STATE  $i_0\leftarrow \dim(H_1)$. $\mathcal{B}_0 \leftarrow \{ \pg_1,\cdots,\pg_{i_0}\}$.
\STATE  Decompose $I$ over $\mathcal{B}_0$ using Algorithm~\ref{alg:decomp}.  
\STATE Decompose $I = (\delta) \prod_{i \leq i_0}\pg_i^{e_i}\cdot \prod_{i > i_0}\pg_i^{f_i}$ 
using Algorithm~\ref{alg:full_q_descent_2}. 
\STATE Deduce $(\delta'\bmod \pg_i)$ and $\Log(\delta')$ such that $I = (\delta')\mathcal{B_0}^{(e_i) + \sum_j f_jH_j}$. 
\STATE $\vec{y}\leftarrow (e_i) + \sum_j f_jH_j$. Solve $\vec{x}H = \vec{y}$. 
\STATE Deduce $(\beta_0\bmod \pg_i)$ and $\Log(\beta_0)$ such that $I = (\beta_0)$. 
\STATE Use the methods of~\cite{CVP_cyclo} to derive $(x_i)$ such that $\beta_0\prod_i u_i^{x_i}$ is a short generator of $I$ where 
the $u_i$ are the cyclotomic units.
\STATE Compute $\beta_0\prod_i u_i^{x_i}\bmod\pg_j$ for $j\leq l$ and reconstruct $\beta =  \beta_0\prod_i u_i^{x_i}$ by the 
Chinese Remainder Theorem.
\RETURN $\beta$
\end{algorithmic}
\end{algorithm}

\begin{proposition}[GRH + Heuristic~\ref{heuristic:smoothness_ideals}]
When using a precomputation of cost $2^{O(2-3a+2\varepsilon)}$ corresponding to Algorithm~\ref{alg:class_group} with smoothness bound 
$B \in 2^{O(2-3a+2\varepsilon)}$ for $a,b,\varepsilon > 0$ such that 
\begin{enumerate}
 \item $2-3a + 2\varepsilon \leq b \leq 7a - 2$, 
 \item $\frac{2}{5} + \frac{\varepsilon}{5} \leq  a \leq \frac{1}{2}$, 
\end{enumerate} 
on the input principal  ideal $I$ satisfying 
$\log(\Nm(I))\leq n^b$, the run time of Algorithm~\ref{alg:gamma_SVP} 
is in $2^{n^{a +o(1)}}$, and it 
returns a short generator of $I$ which is a solution to $\gamma$-SVP for $\gamma\in e^{\tilde{O}(\sqrt{n})}$. 
\end{proposition}

\begin{remark}
We can replace $\varepsilon$ by The constant $\varepsilon > 0$ can be replaced by a value of the form $\frac{\operatorname{Poly}(\log\log(n))}{\log(n)}$. 
In this case, the cost of the precomputation is in $2^{O(2-3a+ o(1))}$.
\end{remark}
For example, with $c= \frac{3}{7}$ and $b=1+o(1)$, if an entity spends a precomputation cost in $2^{(n^{5/7 + \varepsilon})}$ for 
an arbitrarily small $\varepsilon > 0$, then 
all subsequent instances of $\gamma$-SVP (and the search for small generators) in principal ideals $I$ of $\OK$ satisfying $\log(\Nm(I))\leq n^{1+o(1)}$ for $\gamma \in e^{\tilde{O}(\sqrt{n})}$ will 
take heuristic time in $2^{O(n^{3/7+o(1)})}$. In particular, the public keys $g$ of the multilinear maps 
of Garg, Gentry and Halevi satisfy $\|g\|\leq n$, which yields $\Nm(g)\leq 2^{n^{1+o(1)}}$. This means that with a precomputation of cost $2^{(n^{5/7 + \varepsilon})}$, 
our key recovery attack takes heuristic time $2^{O(n^{3/7+o(1)})}$. The attack 
with precomputation can have to main scenarios:
\begin{itemize}
 \item An entity precomputes the data and the attacker downloads it. In this case the storage required is proportional 
 to the expected time of the attack, that is $2^{O(n^a)}$. 
 \item The attacker is allowed to query the entity having done the precomputation. In this case, the entity sends the attacker 
 the significantly smaller matrix corresponding to the relations between the short ideals before hand, and for each challenge, the 
 attacker asks for the decomposition of the large primes in the decomposition of the ideal given as input to the $\qg$-descent. 
\end{itemize}

\begin{remark}
 The precomputation can go on after the computation of the HNF of the large relation matrix. The entity 
 performing the precomputation can continue creating new relations by performing $\qg$-descents on the large 
 primes which are not already in the factor basis $\mathcal{B}$. The HNF of the relation matrix can be updated at 
 a minimal cost. 
\end{remark}

 \begin{remark}
 The private keys in the FHE scheme of Smart and Vertauteren~\cite{fre_smart} are of the form $g = \sum_i a_i\zeta_i$ 
 where $|a_i|\leq 2^{\sqrt{n}}$. In this case, the only bound we have on the norm of the ideal is $\Nm(g)\leq 2^{\tilde{O}(n^{3/2})}$, 
 which means that the PIP algorithm with precomputation presented in this paper would not bring any speedup over the PIP resolution 
 without precomputation (as we need to choose $b = 3/2$ and $a = 1/2$). However, as stated in~\cite{EUROCRYPT17}, the complexity of an 
 attack running in time $2^{n^{1/2+o(1)}}$ expressed in terms of the 
 size $S$ of the key is less than $2^{S^{1/3 + o(1)}}$, which is even better than the complexity we obtain when carrying the 
 attack on the multilinear maps with respect to the size of the keys. Our attack with precomputation allows a speedup with respect 
 to the size of the keys in the case of the multilinear maps since in this context the input size is $S = 2^{n^{1+o(1)}}$ (as opposed 
 to $S = 2^{n^{3/2+o(1)}}$ in the case of the FHE scheme of Smart and Vercauteren). 
\end{remark}

\bibliography{biblio}

\end{document}